\documentclass[11pt,reqno]{amsart}
\usepackage{mathrsfs}
\usepackage[breaklinks]{hyperref}
\usepackage{latexsym}
\usepackage{upref, eucal}

\usepackage[headheight=110pt,top=1in, bottom=.9in, left=0.8in, right=0.8in]{geometry}

\usepackage{url}
\usepackage{amssymb}

\newcommand{\df}{\dfrac}
\newcommand{\tf}{\tfrac}

\renewcommand{\Re}{\operatorname{Re}}

\renewcommand{\a}{\alpha}
\renewcommand{\b}{\beta}

\newcommand{\g}{\gamma}
\newcommand{\G}{\Gamma}
\renewcommand{\l}{\lambda}

\renewcommand{\(}{\left\(}
\renewcommand{\)}{\right\)}
\renewcommand{\[}{\left\[}
\renewcommand{\]}{\right\]}
\renewcommand{\i}{\infty}
\numberwithin{equation}{section}
\theoremstyle{plain}
\newtheorem{theorem}{Theorem}[section]
\newtheorem{lemma}[theorem]{Lemma}
\newtheorem{corollary}[theorem]{Corollary}

\newtheorem{remark}[]{Remark}

\makeatletter
\def\proof{\@ifnextchar[{\@oproof}{\@nproof}}
\def\@oproof[#1][#2]{\trivlist\item[\hskip\labelsep\textit{#2 Proof of\
		#1.}~]\ignorespaces}
\def\@nproof{\trivlist\item[\hskip\labelsep\textit{Proof.}~]\ignorespaces}

\makeatother

\makeatletter
\def\@tocline#1#2#3#4#5#6#7{\relax
	\ifnum #1>\c@tocdepth 
	\else
	\par \addpenalty\@secpenalty\addvspace{#2}%
	\begingroup \hyphenpenalty\@M
	\@ifempty{#4}{%
		\@tempdima\csname r@tocindent\number#1\endcsname\relax
	}{%
		\@tempdima#4\relax
	}%
	\parindent\z@ \leftskip#3\relax \advance\leftskip\@tempdima\relax
	\rightskip\@pnumwidth plus4em \parfillskip-\@pnumwidth
	#5\leavevmode\hskip-\@tempdima
	\ifcase #1
	\or\or \hskip 1em \or \hskip 2em \else \hskip 3em \fi%
	#6\nobreak\relax
	\dotfill\hbox to\@pnumwidth{\@tocpagenum{#7}}\par
	\nobreak
	\endgroup
	\fi}
\makeatother

\allowdisplaybreaks

\begin{document}
	\title[]{Modular relations involving generalized digamma functions}
	\author{Atul Dixit}
	\author{Sumukha Sathyanarayana}
	\author{N. Guru Sharan }
	\address{Discipline of Mathematics, Indian Institute of Technology Gandhinagar, Palaj, Gandhinagar, Gujarat-382355, India}
	\email{adixit@iitgn.ac.in; sumukha.s@iitgn.ac.in; gurusharan.n@iitgn.ac.in}
	\thanks{2020 \textit{Mathematics Subject Classification.} Primary 11M35, 33B15; Secondary 15B05, 41A60.\\
		\textit{Keywords and phrases.} Ramanujan's transformation, generalized digamma functions, Stieltjes constants, modular relation, inversion formula for triangular Toeplitz matrix.}
	
	\begin{abstract}
	Generalized digamma functions $\psi_k(x)$, studied by Ramanujan, Deninger, Dilcher, Kanemitsu, Ishibashi etc., appear as the Laurent series coefficients of the zeta function associated to an indefinite quadratic form. In this paper, a modular relation of the form $F_k(\alpha)=F_k(1/\alpha)$ containing infinite series of $\psi_k(x)$, or, equivalently, between the generalized Stieltjes constants $\gamma_k(x)$,  is obtained for any $k\in\mathbb{N}$. When $k=0$, it reduces to a famous transformation given on page $220$ of Ramanujan's Lost Notebook. For $k=1$, an integral containing Riemann's $\Xi$-function, and corresponding to the aforementioned modular relation, is also obtained along with its asymptotic expansions as $\alpha\to0$ and $\alpha\to\infty$. Carlitz-type and Guinand-type finite modular relations involving $\psi_j^{(m)}(x), 0\leq j\leq k, m\in\mathbb{N}\cup\{0\},$ are also derived, thereby extending previous results on the digamma function $\psi(x)$. The extension of Guinand's result for $\psi_j^{(m)}(x), m\geq2,$ involves an interesting combinatorial sum $h(r)$ over integer partitions of $2r$ into exactly $r$ parts. This sum plays a crucial role in an inversion formula needed for this extension. This formula has connection with the inversion formula for the inverse of a triangular Toeplitz matrix. The modular relation for $\psi_j'(x)$ is subtle and requires delicate analysis.
	\end{abstract}
	\maketitle
	\tableofcontents
	\section{Introduction}\label{intro}
	Let
	\begin{equation}\label{ddash}
		\mathfrak{D'}:=\{\eta: |\arg(\eta)|\leq\pi-\Delta, \Delta>0\}.
	\end{equation} 
	The Hurwitz zeta function $\zeta(z, x)$ is defined \cite[p.~36]{titch} for Re$(z)>1$ and $x\in\mathfrak{D'}$ by 
	$\zeta(z, x):=\sum_{n=0}^{\infty}(n+x)^{-z}$.
	It  can be analytically continued to the entire $z$-complex plane except for a simple pole at $z=1$ with residue $1$, and it is well-known that $\zeta(z, 1)=\zeta(z)$, the Riemann zeta function.
	The Laurent series expansion of $\zeta(z, x)$ around $z=1$ is given by 
	\begin{equation}\label{hurwitzlaurent}
		\zeta( z, x)=\frac{1}{z-1}+\sum_{k=0}^{\infty}\frac{(-1)^k\gamma_k(x)}{k!}(z-1)^k,
	\end{equation}
	where\footnote{Berndt includes the factor $\frac{(-1)^{k}}{k!}$ in the definition of $\gamma_k(z)$ and does not have it in the summand of \eqref{hurwitzlaurent} as has been done here.} Berndt \cite[Theorem 1]{berndthurwitzzeta} showed that
	\begin{equation}\label{scz}
		\gamma_k(x)=\lim_{n\to\infty}\left(\sum_{j=0}^{n}\frac{\log^{k}(j+x)}{j+x}-\frac{\log^{k+1}(n+x)}{k+1}\right).
	\end{equation}
	These $\gamma_k(x)$ called the \emph{generalized Stieltjes constants}. Clearly, letting $x=1$ in \eqref{hurwitzlaurent} and \eqref{scz} respectively gives the Laurent series expansion of $\zeta(z)$ and the corresponding \emph{Stieltjes constants}, that is, $\gamma_k(1)=\gamma_k$.
	
The literature on generalized Stieltjes constants is vast. We mention only a few of the works on the topic. Williams and Zhang \cite[Theorem 1]{williams-zhang} generalized Berndt's aforementioned result, and with the help of which they could estimate $\gamma_k(x)$ for $x\in(0,1]$; see \cite[Theorem 3]{williams-zhang}. Coffey \cite{coffey} gave an addition formula for $\gamma_k(x)$ whose special case \cite[Corollary 2]{coffey} (see \eqref{coffey special} below) we will use in the sequel. Very recently, Pr\'{e}vost \cite[Theorem 2]{prevost} derived a general formula for $\gamma_k(x)$ involving Stirling numbers of the first kind and depending on the degrees of numerator and denominator of the Pad\'{e} approximant used for approximating $e^{-x}$, and, as an application, obtained a new representation for the first Stieltjes constant \cite[Proposition 1]{prevost}, namely, $\gamma_1=\sum_{j=1}^{\infty}\zeta'(2j+1)/(2j+1)$. Blagouchine \cite{blagouchine} undertook a detailed historical survey on $\gamma_k$ and $\gamma_k(x)$, and, in particular, gave an elegant result to compute $\gamma_1(x)$ for $x\in\mathbb{Q}$. A short proof of this result was later given by Chatterjee and Khurana \cite{chatkhur}. There have been several studies on the asymptotic analysis of $\gamma_k(x)$ and on developing numerical methods for their efficient calculation; see, for example, the recent preprint of Tyagi \cite{tyagi} and the references therein. 
	
	Even though the generalized Stieltjes constants have received a lot of attention in the recent years, except for a few studies \cite{ishikan}, \cite{kanemitsuclosed}, they have not been explored much through their connection with higher analogues of the Euler gamma function. This connection is now explained.
	
	For a non-negative integer $k$ and $x\in\mathfrak{D'}$, Dilcher \cite{dilcher} defined the generalized gamma function $\Gamma_k(x)$ by
	\begin{equation*}
		\Gamma_k(x):=\lim_{n\to\infty}\frac{\text{exp}{\left(\frac{\log^{k+1}(n)}{k+1}x\right)}\prod\limits_{j=1}^{n}\text{exp}{\left(\frac{\log^{k+1}(j)}{k+1}\right)}}{\prod\limits_{j=0}^{n}\text{exp}{\left(\frac{\log^{k+1}(j+z)}{k+1}\right)}}.
	\end{equation*}
	Comparing the above limit with the the limit representation of $\Gamma(x)$ \cite[Equation (3.9)]{temme}, it is clear that  $\Gamma_0(x)=\Gamma(x)$. As said by Dilcher \cite[p.~56]{dilcher}, the generalized gamma function $\Gamma_k(x)$ relates to the Stieltjes constant $\gamma_k$ in the same way as the Euler Gamma function $\Gamma(x)$ relates to the Euler constant $\gamma(=\gamma_0)$.  
	
	Dilcher derived various properties of $\Gamma_k(x)$ in \cite{dilcher}. At the end of his paper, he considered the logarithmic derivative of $\Gamma_k(x)$ and proved that \cite[Proposition 10]{dilcher}
	\begin{equation}\label{psikx}
		\psi_k(x):=\frac{\Gamma_k'(x)}{\Gamma_k(x)}=-\gamma_k-\frac{\log^{k}(x)}{x}-\sum_{n=1}^{\infty}\left(\frac{\log^{k}(n+x)}{n+x}-\frac{\log^{k}(n)}{n}\right),
	\end{equation}
	so that $\psi_0(x)=\psi(x)$, the digamma function. Note that letting $x=1$ in \eqref{scz} and replacing $n$ by $n-1$ yields
	\begin{equation*}
		\gamma_k:=\lim_{n\to\infty}\left(\sum_{j=1}^{n}\frac{\log^{k}(j)}{j}-\frac{\log^{k+1}(n)}{k+1}\right).
	\end{equation*}
	Along with \eqref{psikx} and the fact \cite[Lemma 1]{dilcher} that
	\begin{equation*}
		\lim_{n\to\infty}\left(\log^{k+1}(n+x)-\log^{k+1}(n)\right)=0 \hspace{6mm} (x\in\mathfrak{D'}),
	\end{equation*}
	this implies that
	\begin{equation}\label{psigamma}
		\psi_k(x)=-\gamma_k(x).
	\end{equation}
	This was also shown by Shirasaka \cite[p.~136]{shirasaka}. Thus the generalized Stieltjes constant $\gamma_k(x)$ is simply the additive inverse of the logarithmic derivative of $\Gamma_k(x)$, or, in other words, of the \emph{generalized digamma function} $\psi_k(x)$.
	
	Ramanujan briefly treated $\psi_k(x)$ in Entry 22 of Chapter 8 of his second notebook; see \cite{RN_1}. Also, ten years before Dilcher's paper \cite{dilcher} appeared, Deninger \cite{deninger} essentially studied the function $\log\Gamma_1(x)$ in conjunction  with his goal of obtaining an analogue of the Chowla-Selberg formula for real quadratic fields. In his terminology, let $R:\mathbb{R}^{+}\to\mathbb{R}$ denote the function uniquely defined by the difference equation 
	\begin{equation}\label{dendef}
		R(x+1)-R(x)=\log^{2}(x)\hspace{8mm}(R(1)=-\zeta''(0)).
	\end{equation}
	As given in \cite[Remark 2.4]{deninger}, the function $R(x)$ can be analytically continued to $x\in\mathfrak{D'}$. Then from \cite[Equation (2.1)]{dilcher} and \cite[Equation (2.3.1)]{deninger}, it can be observed that Dilcher's $\Gamma_1(x)$ is related to Deninger's $R(x)$ by\footnote{By analytic continuation, Equation (2.3.1) from \cite{deninger} is valid for $x\in\mathfrak{D'}$.}
	\begin{equation*}
		\log(\Gamma_1(x))=\frac{1}{2}(R(x)+\zeta''(0)).
	\end{equation*}
	Languasco and Righi \cite{languasco} have termed $\exp{(R(x))}$ as the \emph{Ramanujan-Deninger gamma function}. For $x>0$, they also give a fast algorithm to compute this function.
	
	Kanemitsu \cite{kanemitsuclosed} extensively studied the higher order analogue of $R(x)$ (briefly alluded to by Deninger in \cite[p.~173]{deninger}), namely, the solution of the difference equation 
	\begin{equation}\label{rx de}
		R_k(x+1)-R_k(x)=\log^{k}(x)\hspace{10mm}(k\in\mathbb{N}\cup\{0\}),
	\end{equation}
	which is convex on some interval $(A, \infty)$, $A>0$,  and satisfies the initial condition $R_k(1)=(-1)^{k+1}\zeta^{(k)}(0)$. In view of \eqref{dendef}, this trivially implies $R_2(x)=R(x)$. He found its Weierstrass representation to be \cite[Equation (2.8)]{kanemitsuclosed}\footnote{the $k$ in front of $\gamma_{k-1}$ is missing in \cite[Equation (2.8)]{kanemitsuclosed}, however, this is corrected in \cite[p.~77]{ishikan}.}
	\begin{equation*}
		R_k(x)=(-1)^{k+1}\zeta^{(k)}(0)-k\gamma_{k-1}x-\log^{k}(x)-\sum_{n=1}^{\infty}\left(\log^{k}(x+n)-\log^k(n)-kx\frac{\log^{k-1}(n)}{n}\right),
	\end{equation*}
	which, in view of \eqref{psikx}, implies \cite[p.~78]{ishikan}
	\begin{equation}\label{psir}
		\psi_k(x)=\frac{1}{k+1}R_{k+1}'(x).	
	\end{equation}
		The functional equation satisfied by $\psi_k(x)$ is \cite[Equation (8.2)]{dilcher}
	\begin{align}\label{dilcher fe}
		\psi_{k}(1+x)=	\psi_{k}(x) +\frac{\log^{k}(x)}{x},
	\end{align}
	which can be easily seen from \eqref{rx de} and \eqref{psir}. It should be noted that $\psi_k(x)$ is analytic in $\mathfrak{D'}$, where $\mathfrak{D'}$ is defined in \eqref{ddash}.
	
	The function $\psi_k(x)$ turns up in a variety of places in number theory. Ishibashi  \cite{ishibashi} evaluated the Laurent series coefficients of a zeta function associated to an indefinite quadratic form in terms of his $k^{\textup{th}}$ order Herglotz function which he defined for $x>0$ by
	\begin{equation*}
		\Phi_k(x):=\sum_{n=1}^{\infty}\frac{k\psi_{k-1}(nx)-\log^{k}(nx)}{n}.
	\end{equation*}
	The Herglotz-Zagier function $\Phi_1(x)$ and its various generalizations such as the above are themselves very useful in number theory, in particular, in Kronecker limit formulas. See \cite{hhf1} and the references therein.  
	
	Recently, Banerjee, Gupta and the first author encountered an infinite series consisting of $\psi(x)$ and $\psi_1(x)$ while obtaining an explicit transformation \cite[Theorem 1.1]{bdg_log} for the Lambert series of logarithm, that is, for $\sum\limits_{n=1}^{\infty}\displaystyle\frac{\log(n)}{e^{ny}-1}$, where Re$(y)>0$. Using this transformation, they found its asymptotic expansion as $y\to0$ and further used it to find the complete asymptotic expansion of the integral
	\begin{equation*}
		\int_{0}^{\infty}\zeta\left(\frac{1}{2}-it\right)\zeta'\left(\frac{1}{2}+it\right)e^{-\delta t}\, dt
	\end{equation*}
	as $\delta\to0$.
	
	In order to be able to state the main goal of this paper, we now turn our attention to a beautiful modular relation found on page $220$ of Ramanujan's Lost Notebook \cite{bcbad}.
	
	Let\footnote{By analytic continuation, the result holds more generally for Re$(\alpha)>0$ and Re$(\beta)>0$.} $\alpha,\ \beta>0$ be such that $\alpha\beta=1$, and, for $x>0$, let
	\begin{equation}\label{ihp}
		\phi(x):=\psi(x)+\frac{1}{2x}-\log(x).
	\end{equation}
	Then Ramanujan claimed that
	\begin{align}\label{w1.26}
		\sqrt{\alpha}\left\{\dfrac{\gamma-\log(2\pi\alpha)}{2\alpha}+\sum_{n=1}^{\i}\phi(n\alpha)\right\}
		&=\sqrt{\beta}\left\{\df{\gamma-\log(2\pi\beta)}{2\beta}+\sum_{n=1}^{\i}\phi(n\beta)\right\}\nonumber\\
		&=-\df{1}{\pi^{3/2}}\int_0^{\i}\left|\Xi\left(\df{1}{2}t\right)\Gamma\left(\df{-1+it}{4}\right)\right|^2
		\df{\cos\left(\tf{1}{2}t\log\alpha\right)}{1+t^2}\, dt,
	\end{align}
	where
	\begin{align*}
		\Xi(t)&=\xi\left(\frac{1}{2}+it\right),\\
		\xi(s)&=\frac{1}{2}s(s-1)\pi^{-\frac{s}{2}}\Gamma\left(\frac{s}{2}\right)\zeta(s),
	\end{align*}
 are Riemann's functions \cite[p.~16]{titch}.
	
	By a \emph{modular relation}, we mean a transformation of the form $F(-1/z)=F(z)$, where $z\in\mathbb{H}$ (the upper-half plane). The $F$ in such a relation may not be governed by $z\to z+1$. Any such transformation can be equivalently put in the form $\mathscr{F}(\alpha)=\mathscr{F}(\beta)$, where $\alpha, \beta$ are such that Re$(\alpha)>0$, Re$(\beta)>0$ and $\alpha\beta=1$. See \cite{dixitms} for more details.
	
	Recently, Gupta and Kumar \cite{gk} showed that \eqref{w1.26} lives in the realms of the theory of Herglotz function $\Phi_1(x)$. For a comprehensive survey on various extensions and generalizations of this transformation, the reader is referred to \cite{dixzah}.
	
	Since $\psi_0(x)=\psi(x)$, a natural question that arises now is, \emph{do the generalized digamma functions $\psi_k(x), k>0$, also satisfy a modular relation of the form in \eqref{w1.26}}? 
	
	We answer this question affirmatively in Theorem \ref{ramanujan generalization psi_k infinite}. The interesting thing here is that for $k>0$, the modular relation involves not just an infinite series of $\psi_k$ but rather a weighted finite sum involving infinite series of $\psi_0, \psi_1, \cdots \psi_k$. Obtaining a modular relation involving just $\psi_k$ for $k>0$ looks inconceivable.  For $k=1$, we also obtain an analogue of the integral involving the Riemann $\Xi(t)$ function in \eqref{w1.26} (see Theorem \ref{ramanujan analogue psi_1 infinite}). In principle, we can obtain such an integral for a general $k>1$ as well, however, we refrain from doing so as the form of the integrand soon becomes unwieldy. 
	
	It is known that the integral in \eqref{w1.26} admits the following asymptotic expansion as $\alpha\to\infty$, namely, \cite[Equation (1.5)]{oloa} (see also \cite[Corollary 1.7]{dixitkumarhrj} with $w=0$ or \cite[Theorem 1.10]{drzacta} with $z=0$)
	\begin{align}\label{ramanujan asymptotic}
		&\dfrac{1}{\pi^{3/2}}\int_0^{\i}\left|\Xi\left(\dfrac{1}{2}t\right)\Gamma\left(\dfrac{-1+it}{4}\right)\right|^2
		\dfrac{\cos\left(\tf{1}{2}t\log\alpha\right)}{1+t^2}\, dt\nonumber\\
		&\sim-\sqrt{\alpha}\left(\dfrac{\gamma-\log(2\pi\alpha)}{2\alpha}\right)-2\sqrt{\alpha}\sum_{k=1}^{\infty}\frac{(-1)^k}{(2\pi\alpha)^{2k}}\Gamma(2k)\zeta^2(2k).
	\end{align}
	In Theorem \ref{asymptotic_xit}, we obtain the corresponding asymptotic expansion of the integral involving the Riemann $\Xi$-function associated with the modular relation involving $\psi(x)$ and $\psi_1(x)$.
	
	In  \cite{apg3}, Guinand obtained modular relations for  higher derivatives of $\psi(x)$. Indeed,  as shown in \cite[Equation (3.24)]{dixitijnt}, Equation (9) from \cite{apg3} can be written in the form
	\begin{align}\label{guigen}
		\alpha^{\frac{z}{2}}\sum_{j=1}^{\infty}\psi^{(z-1)}(1+j\alpha)=\beta^{\frac{z}{2}}\sum_{j=1}^{\infty}\psi^{(z-1)}(1+j\beta),
	\end{align}
	where $\alpha\beta=1$ and $z\in\mathbb{N}, z>2$. Guinand \cite[p.~4]{apg3} also derived the modular relation involving the first derivative of $\psi(x)$ which can be again reformulated in the form
	\begin{equation}\label{guigen1}
		\alpha\sum_{j=1}^{\infty}\left(\psi'(1+j\alpha)-\frac{1}{j\alpha}\right)-\frac{1}{2}\log(\alpha)=	\beta\sum_{j=1}^{\infty}\left(\psi'(1+j\beta)-\frac{1}{j\beta}\right)-\frac{1}{2}\log(\beta),
	\end{equation}
	where, again, $\alpha\beta=1$. 
	
	The second goal of this paper is obtain generalizations of \eqref{guigen} and \eqref{guigen1} for $\psi_k$ for $k>0$. For a natural number $z$ greater than $2$, the generalization of \eqref{guigen} derived in Theorem \ref{guinand gen psi_k z>2} involves the interesting function $h(r)$ which is a sum over all non-negative integer solutions of $1 b_1+...+ (r+1)b_{r+1}= 2r,  b_1+...+b_{r+1}=r$, that is, over all integer partitions of $2r$ into $r$ parts. In fact, this function turns up in a fundamental inversion formula that is of independent interest. See Theorem \ref{Inv}. Even though this formula seems to be equivalent to the inversion formula for the inverse of a triangular Toeplitz matrix \cite[Corollary 1]{merca}, the version we derive here is, to the best of our knowledge, new. Our proof of this inversion formula is also new. In the course of proving results such as Theorem \ref{guinand gen psi_k z>2}, the inversion formula also gives novel results such as (see \eqref{Inv6} below)
		\begin{align*}
		\zeta^{(\ell)}(n,x) &= \sum_{r=0}^{\ell} \frac{(-1)^{\ell+r+1}}{s^{r+1}(n,1)} \frac{\ell!}{(\ell-r)!} h(r) \psi_{\ell-r}^{(n-1)} (x),
	\end{align*}
	where $n\in\mathbb{N}, n>1$, $\ell\in\mathbb{N}\cup\{0\}$, $h(r)$ is defined in \eqref{h(r)}, and $s(n, m)$ denotes the Stirling number of the first kind.
	 The special case $\ell=0$ of this result is well-known; see, for example, \cite[Equation (3.22)]{dixitijnt}.
	
	The generalization of \eqref{guigen1} given in Theorem \ref{curious} corresponds to the ``$z=2$'' case of Theorem \ref{guinand gen psi_k z>2}. Of course, the series in Theorem \ref{guinand gen psi_k z>2} upon letting $z=2$ do not converge, and one has to separately handle this delicate case. This curious case requires intermediate results which are of independent interest, for example, see lemmas \ref{lemma minus two}, \ref{8.2a} and \ref{estimatelog}.
	
	The third goal of this paper is to give finite analogues of all of the above results. Guinand \cite{apg4} obtained some finite identities related to Poisson's summation formula. Carlitz \cite{carlitz0} found an easier way to obtain finite identities of Guinand type by considering a function that satisfies the multiplication formula
	\begin{equation}\label{carlitz1.4}
		\sum_{j=0}^{n-1}f\left(x+\frac{j}{n}\right)=C_nf(nx)\hspace{7mm}(n\in\mathbb{N}),
	\end{equation}
	where $C_n$ is independent of $x$ but may depend on $f$. As shown by Carlitz in one of his other papers \cite{carlitz}, the above equation implies the symmetric formula
	\begin{equation*}
		C_n\sum_{j=0}^{m-1}f\left(nx+\frac{nj}{m}\right)=	C_m\sum_{j=0}^{n-1}f\left(mx+\frac{mj}{n}\right),
	\end{equation*}
	Since $\zeta(z, x)$ satisfies \eqref{carlitz1.4} with $C_n=n^z$, Carlitz derived the transformation \cite[Equation (3.10)]{carlitz0}
	\begin{equation}\label{carlitz relation}
		n^z\sum_{j=1}^{m}\zeta\left(z, nx+\frac{n(j-1)}{m}\right)=	m^z\sum_{j=1}^{n}\zeta\left(z, mx+\frac{m(j-1)}{n}\right).
	\end{equation}
	This transformation is our starting point towards obtaining finite analogues of our aforementioned modular relations involving infinite series of $\psi_k, k\geq0,$ and their derivatives. See Theorems \ref{finitetrans2}, \ref{finitetrans1} and their corollaries.
	
	\section{Main results}\label{mr}
	
	The generalization of \eqref{w1.26}, wherein the digamma function $\psi(x)$ is replaced by the Hurwitz zeta function $\zeta(z, x)$, and which plays a key role in obtaining some of our results, is the following one from \cite[Theorem 1.4]{dixitijnt}:\\
	
	\noindent
	 Let $0<$ \textup{Re} $z<2$. Define $\varphi(z,x)$ by
	\begin{equation*}
		\varphi(z,x):=\zeta(z,x)-\frac{1}{2}x^{-z}+\frac{x^{1-z}}{1-z},
	\end{equation*}
	If $\alpha$ and $\beta$ are such that Re$(\alpha)>0$ and Re$(\beta)>0$\footnote{Even though the result is stated only for positive numbers $\alpha$ and $\beta$, it can be easily seen to be true for Re$(\alpha)>0$ and Re$(\beta)>0$ by analytic continuation.} such that $\alpha\beta=1$, then
	\begin{align}\label{mainneq2}
		\a^{\frac{z}{2}}\left(\sum_{n=1}^{\infty}\varphi(z,n\a)-\frac{\zeta(z)}{2\alpha^{z}}-\frac{\zeta(z-1)}{\alpha (z-1)}\right)&=\b^{\frac{z}{2}}\left(\sum_{n=1}^{\infty}\varphi(z,n\b)-\frac{\zeta(z)}{2\beta^{z}}-\frac{\zeta(z-1)}{\beta (z-1)}\right)\nonumber\\
		&=\frac{8(4\pi)^{\frac{z-4}{2}}}{\Gamma(z)}\int_{0}^{\infty}\omega(z,t)\cos\left(\tf{1}{2}t\log\a\right)\, dt,
	\end{align}
	where
	\begin{equation}\label{omdef}
		\omega(z,t):=\frac{1}{(z^2+t^2)}\Gamma\left(\frac{z-2+it}{4}\right)\Gamma\left(\frac{z-2-it}{4}\right)
		\Xi\left(\frac{t+i(z-1)}{2}\right)\Xi\left(\frac{t-i(z-1)}{2}\right).
	\end{equation}
	The integral in \eqref{mainneq2} was considered by Ramanujan \cite{riemann} who obtained alternate integral representations for it in different regions of the $z$-complex plane. Regarding the special case $z=1$ of this integral, Hardy \cite{ghh} said, \textit{``the properties of this integral resemble those of one which Mr. Littlewood and I have used, in a paper to be published shortly in Acta Mathematica to prove that\footnote{There is a typo in this formula in that $\pi$ should not be present.}}
	\begin{equation*}
		\int_{-T}^{T}\left|\zeta\left(\frac{1}{2}+ti\right)\right|^2\, dt \sim
		\frac{2}{\pi} T\log T\hspace{3mm}(T\to\infty)\textup{''}.
	\end{equation*}
	Now \eqref{hurwitzlaurent} and \eqref{psigamma} suggest that differentiating both sides of the first equality in \eqref{mainneq2} $k$ times with respect to $z$ and then letting $z\to1$ might lead to a modular relation involving the generalized digamma functions. That this is indeed the case is shown in the following sub-section.
	
	\subsection{An extension of Ramanujan's transformation for generalized digamma functions}
	
	Before giving the main result, we record the following asymptotic expansion of $\psi_j(x)$ as $x \to \infty$,  Re$(x)>0$, given in Proposition 3 of \cite{coffey}:
	\begin{align}\label{asymptotic coffey}
		\psi_j(x)\sim\frac{\log^{j+1}(x)}{j+1}-\frac{1}{2x}\log^{j}(x)+\sum_{m=1}^{\infty}\frac{B_{2m}}{(2m)!}x^{-2m}\sum_{t=0}^{j}\binom{j}{t}t!s(2m,t+1)\log^{j-t}(x),
	\end{align}
	where $B_m$ denotes $m^{\textup{th}}$ Bernoulli number.
	\begin{theorem}\label{ramanujan generalization psi_k infinite}
		For \textup{Re}$(x)>0$, define
		\begin{align}\label{ramanujan generalization psi_k infinite eqn}
			\mathscr{F}_k(x):=\sqrt{x} \sum_{j=0}^{k}   (-1)^{j+1} \binom{k}{j} \log^{k-j}(\sqrt{x})  &\left\{\sum_{n=1}^{\infty} \left(\psi_j(nx)+\frac{\log^j(nx)}{2 n x}- \frac{\log^{j+1}(nx)}{j+1}\right) \right. \notag \\
			&\left. \hspace{-5mm}+\sum_{\ell=0}^{j}   \binom{j}{\ell} \frac{\gamma_\ell\log^{j-\ell}(x) }{2  x}  -\frac{\log^{j+1}(x)+ 2(-1)^{j+1} \zeta^{(j+1)}(0)}{2 (j+1) x}	\right\}.
		\end{align}
		Then for any $\alpha$, $\beta$  with \textup{Re}$(\alpha)>0$, \textup{Re}$(\beta)>0$ and $\alpha\beta=1$, we have
		\begin{equation}
				\mathscr{F}_k(\alpha)=	\mathscr{F}_k(\beta). \label{Ramanujan Identity}
		\end{equation}
	\end{theorem}
	When $k=1$, we explicitly obtain and state in the following theorem the integral containing the Riemann $\Xi$-function equal to each of the expressions in the above modular relation.
	\begin{theorem}\label{ramanujan analogue psi_1 infinite}
		Let \textup{Re}$(x)>0$. Let $\phi(x)$ be defined in \eqref{ihp}. Moreover, define 
		\begin{align*}
			\phi_1(x)&:= \psi_1(x) + \frac{1}{2x}\log(x)- \frac{1}{2}\log^2(x),
		\end{align*}
		\begin{align}\label{ef1}
				\mathscr{F}_1(x)&:=\sqrt{x} \left\{ \sum_{n=1}^{\infty} \phi_1(nx) + \frac{\log^2(2\pi)-(\gamma-\log(x))^2}{4x} + \frac{\pi^2}{48x} \right\} -\frac{\sqrt{x} \log(x)}{2} \left\{ \sum_{n=1}^{\infty} \phi(nx) +\frac{\gamma - \log (2\pi x)}{2x} \right\},
		\end{align}
		and
		\begin{align}\label{mathscr i alpha}
			\mathscr{I}(x)&:= \frac{2}{(4 \pi)^{\frac{3}{2}}}  \int_{0}^{\infty} \left| \Gamma\left(\frac{-1+i t}{4}\right)  \Xi\left(\frac{t}{2} \right) \right|^2  \notag\\
			&\qquad\qquad\quad\times\left\{  \psi \left( \frac{-1+it}{4} \right) +  \psi \left( \frac{-1-it}{4} \right) - \frac{8}{1+t^2} + 2\log(4\pi) + 4\gamma \right\} \frac{\cos\left(\frac{1}{2}t \log x\right)}{1+t^2} dt.
		\end{align}
		For $\alpha,\beta $ such that \textup{Re}$(\alpha)>0$, \textup{Re}$(\beta)>0$ and $\alpha \beta =1$, we have
		\begin{align}\label{k=1}
				\mathscr{F}_1(\alpha)=	\mathscr{F}_1(\beta)=\mathscr{I}(\alpha)=\mathscr{I}(\beta).
		\end{align}
	\end{theorem}
	\begin{remark}
	For the expression on the right-hand side of \eqref{ef1}, we use the notation $\mathscr{F}_1$ because it is the $k=1$ case of the function $\mathscr{F}_k$ defined in \eqref{ramanujan generalization psi_k infinite eqn}, as can be seen by employing \eqref{zeta''0}.
	\end{remark}
	\subsection{Asymptotic expansion of a new $\Xi$-function integral}
	
	\begin{theorem}\label{asymptotic_xit}
		Let $\mathscr{I}(\alpha)$ be defined in \eqref{mathscr i alpha}. Let $\mathfrak{D}=\{\eta\in\mathbb{C}: |\arg(\eta)|<\pi/2\}$.  As $\alpha\to\infty$ along any path in $\mathfrak{D}$,
		\begin{align}\label{alpha infinity}
			\mathscr{I}(\alpha)\sim& \frac{-1}{4\sqrt{\alpha}} \left( \gamma + \log(2\pi) \right) \left( \gamma- \log(2\pi\alpha) \right) +\frac{\pi^2}{48\sqrt{\alpha}}\notag \\
			&+ 2\sqrt{\alpha} \sum_{m=0}^{\infty} \frac{(-1)^{m} \Gamma(2m+2) \zeta^2(2m+2)}{(2\pi\alpha)^{2m+2}} \left(- \frac{1}{2}\log(\alpha) +\gamma+ \psi(2m+2) + \frac{\zeta'(2m+2)}{\zeta(2m+2)} \right).
		\end{align}
		Moreover, the invariance $\mathscr{I}(\alpha)=\mathscr{I}(1/\alpha)$ shows that as $\alpha\to0$, $\mathscr{I}(\alpha)$ is asymptotic  to the expression obtained by replacing $\alpha$ on the right-hand side of \eqref{alpha infinity} by $1/\alpha$.
	\end{theorem}	
	The above result has a nice application in that it can be used to get the asymptotic behavior of $ \sum_{n=1}^{\infty} \phi_1(nx) $ as $x\to0$ and $x\to\infty$, which is recorded in the following corollary.
	\begin{corollary}\label{phi1 asymptotic}
	Let $\phi_1(x)$ and $\mathfrak{D}$  be defined in Theorems \ref{ramanujan analogue psi_1 infinite} and \ref{asymptotic_xit} respectively. As $\alpha\to0$ along any path in $\mathfrak{D}$,
	\begin{align}\label{alpha zero}
		\sum_{n=1}^{\infty}\phi_1(n\alpha)&\sim\frac{\pi^2}{48}\left(1-\frac{1}{\alpha}\right)-\frac{1}{4\alpha}\left(\gamma+\log\left(\tfrac{2\pi}{\alpha}\right)\right)\left(\alpha\left(\gamma-\log\left(\tfrac{2\pi}{\alpha}\right)\right)-\gamma+\log(2\pi\alpha)\right)\notag\\
		&\quad+2\sum_{m=0}^{\infty} \frac{(-1)^{m} \Gamma(2m+2) \zeta^2(2m+2)}{(2\pi)^{2m+2}} \left(\gamma+ \psi(2m+2) + \frac{\zeta'(2m+2)}{\zeta(2m+2)} \right)\alpha^{2m+1}.
	\end{align}
	Also, as $\alpha\to\infty$ along any path in $\mathfrak{D}$,
	\begin{align}\label{infinity alpha}
	\sum_{n=1}^{\infty}\phi_1(n\alpha)\sim2 \sum_{m=0}^{\infty} \frac{(-1)^{m} \Gamma(2m+2) \zeta^2(2m+2)}{(2\pi\alpha)^{2m+2}} \left(-\log(\alpha) +\gamma+ \psi(2m+2) + \frac{\zeta'(2m+2)}{\zeta(2m+2)} \right).
	\end{align}
	\end{corollary}
	\subsection{Carlitz-type transformations}
	
	\begin{theorem} \label{finitetrans2}
		Let $\mathfrak{D'}$ be defined in \eqref{ddash}. 
		For $x\in\mathfrak{D'}$, $k\in\mathbb{N}\cup\{0\}$ and $m,n \in \mathbb{N}$,
		\begin{align} \label{finitetrans2_eqn}
			&n \left\{  \sum_{\ell=0}^{k}  (-1)^{\ell} \binom{k}{\ell} \log^{k-\ell}(n)  \sum_{j=1}^{m} \psi_\ell \left(nx + \frac{n(j-1)}{m}\right) - \frac{1}{k+1} m \log^{k+1}(n) \right\} \notag\\
			=&m \left\{ \sum_{\ell=0}^{k} (-1)^{\ell} \binom{k}{\ell} \log^{k-\ell}(m)  \sum_{j=1}^{n}  \psi_\ell\left(mx + \frac{m(j-1)}{n}\right) - \frac{1}{k+1} n \log^{k+1}(m) \right\}.
		\end{align}
	\end{theorem} 
	
	\begin{corollary} \label{carlitzcorollary}
		Let $x\in\mathfrak{D'}$. For any $m, n \in \mathbb{N}$,
		\begin{align*}
			n \left[ \sum_{j=1}^{m} \psi \left(nx + \frac{n(j-1)}{m}\right) + m\log (m) \right] =m  \left[ \sum_{j=1}^{n} \psi  \left(mx + \frac{m(j-1)}{n}\right) + n \log (n) \right].
		\end{align*} 
	\end{corollary}
	The above corollary can be viewed as a finite analogue of Ramanujan's transformation \eqref{w1.26}. Guinand essentially derived it; see \cite[Equation (8)]{apg3} for more details.
	\subsection{A fundamental inversion formula}
	
	Inversion formulas are extremely useful in combinatorics and number theory. In what follows, we state a fundamental inversion formula which will help us realize our objective of finding modular relations between derivatives of a fixed order of $\psi(x), \psi_1(x), \cdots, \psi_k(x)$ for every $k\in\mathbb{N}$.    
	\begin{theorem} \label{Inv}
		Let f, g and s be arithmetic functions with $s(1)\neq0$ such that for $k \in \mathbb{N} \cup \{0\} $,
		\begin{align*}
			g(k)= \sum_{r=0}^{k} s(k-r+1) f(r).
		\end{align*}
		Then,
		\begin{align*}
			f(k)= \sum_{r=0}^{k}  \frac{(-1)^{r}}{s^{r+1}(1)} h(r) g(k-r),
		\end{align*}
		where
		\begin{align}\label{h(r)}
			h(r)= \sum (-1)^{b_1} \prod_{i=1}^{r+1} s^{b_i}(i) \frac{(r-{b_1})! }{b_2!...b_{r+1}!},
		\end{align}
		where the sum runs over all non-negative integer solutions of $1 b_1+...+ (r+1)b_{r+1}= 2r, \\ b_1+...+b_{r+1}=r$.
	\end{theorem}
	The proof is a nice application of mathematical induction.
	
	\subsection{Modular relations involving the derivatives of first $k$ generalized digamma functions}
	
	\begin{theorem} \label{finitetrans1}
		Let $x\in\mathfrak{D'}$, $z\in\mathbb{N}, z>1$, and $m,n\in\mathbb{N}$. For any non-negative integer $k$,
		\begin{align}\label{meeting equation}
			&\sum_{\ell=0}^{k} \frac{2^\ell}{(k-\ell)!} \left(\frac{m}{n}\right)^{\frac{-z}{2}} \log^{k-\ell}\left(\frac{m}{n}\right) \sum_{r=0}^{\ell} \frac{(-1)^r }{(\ell-r)!} \frac{h(r)}{s^r(z,1)} \sum_{j=1}^{m}  \psi_{\ell-r}^{(z-1)} \left(nx + \frac{n(j-1)}{m}\right)\notag\\
			=&\sum_{\ell=0}^{k} \frac{2^\ell}{(k-\ell)!} \left(\frac{n}{m}\right)^{\frac{-z}{2}} \log^{k-\ell}\left(\frac{n}{m}\right)  \sum_{r=0}^{\ell} \frac{(-1)^r }{(\ell-r)!} \frac{h(r)}{s^r(z,1)}  \sum_{j=1}^{n} \psi_{\ell-r}^{(z-1)} \left(mx + \frac{m(j-1)}{n}\right),
		\end{align}
		where $h(r)$, defined in Theorem \ref{Inv}, is considered with $s(i)=s(z, i)$, the Stirling number of the first kind.
	\end{theorem}
	Two special cases of Theorem \ref{finitetrans1} are stated next. These can be considered to be analogues of the duplication formula of $\psi(x)$ \cite[p.~76]{temme}.
	\begin{corollary} \label{finitetrans1 cor1}
		For $x\in\mathfrak{D'}$,
		\small\begin{align*}
			\left( 1 - \frac{1}{2} \log2 \right) \psi'(2x) + \psi_1'(2x) = \frac{1}{4} \left\{ \left( 1 + \frac{1}{2} \log2 \right) \left( \psi'(x) + \psi'\left(x+\frac{1}{2}\right) \right) + \left( \psi_1'(x) + \psi_1'\left(x+\frac{1}{2}\right) \right) \right\}.
		\end{align*}
	\end{corollary}
	
	\begin{corollary} \label{finitetrans1 cor2}
		For $x\in\mathfrak{D'}$,
		\small\begin{align*}
			&\left( \frac{1}{4} \log^2 2 + \log2 +2 \right)  \left\{ \psi'(x) + \psi'\left(x + \frac{1}{2}\right) \right\} + 	\left( \log2 +2 \right) \left\{ \psi_1'(x) + \psi_1'\left(x + \frac{1}{2}\right) \right\} + 	\left\{ \psi_2'(x) + \psi_2'\left(x + \frac{1}{2}\right) \right\} \\
			&= 4 \left[ \left( \frac{1}{4} \log^2 2 - \log2 +2 \right)\psi'(2x) + \left( - \log2 +2 \right) \psi_1'(2x)+ \psi_2'(2x) \right].
		\end{align*}
	\end{corollary}
	
	\subsection{Guinand-type transformations}
	
	\subsubsection{The case $z\in\mathbb{N}, z>2$}
	
	\begin{theorem}\label{guinand gen psi_k z>2}
		Fix $z\in\mathbb{N}, z>2$. For $\alpha, \beta>0$ such that $\alpha \beta =1$, we have
		\begin{align}\label{z>2}
			&\sum_{\ell=0}^{k} \frac{2^\ell}{(k-\ell)!} \alpha^{\frac{-z}{2}} \log^{k-\ell}(\alpha) \sum_{r=0}^{\ell} \frac{(-1)^r }{(\ell-r)!} \frac{h(r)}{s^r(z,1)} \sum_{j=1}^{\infty}  \psi_{\ell-r}^{(z-1)} \left(1 + \frac{j}{\alpha}\right)\notag\\
			=&\sum_{\ell=0}^{k} \frac{2^\ell}{(k-\ell)!} \beta^{\frac{-z}{2}} \log^{k-\ell}(\beta)  \sum_{r=0}^{\ell} \frac{(-1)^r }{(\ell-r)!} \frac{h(r)}{s^r(z,1)} \sum_{j=1}^{\infty} \psi_{\ell-r}^{(z-1)} \left(1 + \frac{j}{\beta}\right),
		\end{align}
		where $h(r)$, defined in Theorem \ref{Inv}, is considered with $s(i)=s(z, i)$, the Stirling number of the first kind.
	\end{theorem}
	
	The infinite version of Theorem \ref{finitetrans1} for $z=2$ cannot be directly obtained in a similar way as the one obtained above for $z\in\mathbb{N}, z>2$ because the associated infinite series involving $\psi_j'$ diverge. By a careful limiting process, however, one can obtain the following result.
	\subsubsection{The curious case of $z=2$}
	
	\begin{theorem}\label{curious}
For $\alpha, \beta>0$ such that $\alpha\beta=1$, 
		\small\begin{align}
			&\sum_{\ell=0}^{k}\frac{1}{\ell!} \sum_{r=0}^{k-\ell}  \frac{1}{r!} \left\{ (-1)^\ell 2^{k-\ell}  \alpha \log^{\ell} \left( \alpha \right)  \sum_{j=1}^{\infty} \left(  \psi_{r}' \left(1 + \alpha j\right) - \frac{\log^{r}(\alpha j)}{\alpha j} \right)  + a_\ell 2^{k-r} \log^{r} \left( \alpha \right)  \right\} - \frac{1}{2(k+1)!}  \log^{k+1}\left( \alpha \right) \notag\\
				&\sum_{\ell=0}^{k}\frac{1}{\ell!} \sum_{r=0}^{k-\ell}  \frac{1}{r!} \left\{ (-1)^\ell 2^{k-\ell}  \beta \log^{\ell} \left( \beta \right)  \sum_{j=1}^{\infty} \left(  \psi_{r}' \left(1 + \beta j\right) - \frac{\log^{r}(\beta j)}{\beta j} \right)  + a_\ell 2^{k-r} \log^{r} \left( \beta \right)  \right\} - \frac{1}{2(k+1)!}  \log^{k+1}\left( \beta \right),\label{curious eqn}
		\end{align}
		\normalsize
		where $a_{\ell}$ is $\gamma-1$ if $\ell=0$, and $\gamma_\ell$ otherwise.
	\end{theorem}
	
	\section{Preliminaries}
	
	Throughout the paper, we will be using the notation $\zeta^{(j)}(z,x)$ to denote the $j$th derivative of $\zeta(z, x)$ with respect to $z$. Stirling's formula for $\Gamma(\sigma+it)$ in the vertical strip $p\leq\sigma\leq q$ states that as $|t|\to \infty$, we have \cite[p.~224]{cop}
	\begin{equation}\label{strivert}
		|\Gamma(\sigma+it)|=\sqrt{2\pi}|t|^{\sigma-\frac{1}{2}}e^{-\frac{1}{2}\pi |t|}\left(1+O\left(\frac{1}{|t|}\right)\right).
	\end{equation}
Watson's lemma \cite[p.~71]{olver} is given by
	\begin{theorem}\label{watlem}
		If $q(t)$ is a function of the positive real variable $t$ such that
		\begin{equation*}
			q(t)\sim\sum_{s=0}^{\infty}a_st^{(s+\lambda-\mu)/\mu}\hspace{3mm} (t\to 0)
		\end{equation*}
		for positive constants $\lambda$ and $\mu$, then
		\begin{equation*}
			\int_{0}^{\infty}e^{-yt}q(t)\, dt\sim\sum_{s=0}^{\infty}\Gamma\left(\frac{s+\lambda}{\mu}\right)\frac{a_s}{y^{(s+\lambda)/\mu}}\hspace{3mm} (y\to\infty),
		\end{equation*}
		provided that this integral converges throughout its range for all sufficiently large $y$.
	\end{theorem}
	The above result also holds \cite[p.~32]{temme} for complex $\lambda$ with Re$(\lambda)>0$, and for $y\in\mathbb{C}$ with the integral being convergent for all sufficiently large values of Re$(y)$. 
	
	The following generalization of Watson's lemma which will also be needed in the sequel is due to Wong and Wyman \cite[Theorem 4.1]{wong-wyman}.
	\begin{theorem} \label{watsongen}
		For $\gamma\in\mathbb{R}$, define a function 
		\begin{align*}
			F(z):=\int_{0}^{\infty e^{i\g}}f(t)e^{-zt}dt. 
		\end{align*}
		Assume that $F(z)$ exists for some $z=z_0$. If 
		\begin{enumerate}
			\item for each integer $N\in \mathbb{N}\cup\{0\}$ 
			\begin{align*}
				f(t)=\sum_{n=0}^{N}c_nt^{\l_n-1}P_n(\log t)+o(t^{\l_N-1}(\log t)^{m(N)}),
			\end{align*}
			as $t\to 0$ along $\arg(t)=\g$.
			\item $P_n(\omega)$ is a polynomial of degree $m=m(n)$.
			\item $\{\l_n\}$ is a sequence of complex numbers, with $\Re (\l_{n+1})>\Re( \l_{n}), \Re (\l_{0})>0,$ for all $n$ such that $n$ and $n+1$ are in $\mathbb{N}\cup\{0\}$.
			\item $\{c_n\}$ is a sequence of complex numbers. 
		\end{enumerate}
		Then as $z\to \infty $ in $S(\Delta)$,
		\begin{align*}
			F(z)\sim \sum_{n=0}^{N} c_nP_n(D_n)[\Gamma(\l_n)z^{-\l_n}]+o\left(z^{-\l_N}(\log z)^{m(N)}\right),
		\end{align*}
		where $S(\Delta): |\arg (ze^{i\g})|\leqq{\pi\over 2}-\Delta$, and $D_n$ is the operator $D_n:={d\over d\l_n}$. This result is uniform in the approach of $z\to \infty$ in $S(\Delta)$.
	\end{theorem}

	\section{Proofs of the extension of Ramanujan's formula \eqref{w1.26} and its analogue}

 \begin{proof}[Theorem \textup{\ref{ramanujan generalization psi_k infinite}}][]
	We first prove the result for $\alpha, \beta>0$ and then extend it to Re$(\alpha)>0$, Re$(\beta)>0$ by analytic continuation. For any positive real number $x$ and $0<\textup{Re}(z)<2$, let us denote  
	\begin{align}
		F(z,x) := \sum_{n=1}^{\infty} \varphi(z, n x)-\frac{\zeta(z)}{2 x^z}-\frac{\zeta(z-1)}{x (z-1)} \label{F(x)} ,
	\end{align}
	where $\varphi(z,x)= \zeta(z,x)-\frac{1}{2} x^{-z}+\frac{x^{1-z}}{1-z}$, for any $x \in \mathbb{R}^{+}$, so that from \eqref{mainneq2}, for $\alpha \beta =1$,
	\begin{align}
	 \alpha^{\frac{z}{2}} F(z,\alpha)=\beta^{\frac{z}{2}} F(z,\beta).\label{IJNT-Ramanujan}
	\end{align}
	Differentiate \eqref{IJNT-Ramanujan} $k$ times with respect to $z$ and the let $z \rightarrow 1$, thereby obtaining
	\begin{align}
		\lim_{z \to 1} \frac{d^k}{dz^k} \left(\alpha^{\frac{z}{2}} F(z,\alpha) \right) =
		\lim_{z \to 1} \frac{d^k}{dz^k} \left(\beta^{\frac{z}{2}} F(z,\beta) \right) \label{diffIJNT} . 
	\end{align}
	By Leibnitz rule,
	\begin{align}
		\lim_{z \to 1} \frac{d^k}{dz^k} \left(x^{\frac{z}{2}} F(z,x) \right) = 	\sqrt{x}	\sum_{j=0}^{k} \binom{k}{j} \log^{k-j}(\sqrt{x})\lim_{z \to 1} \frac{d^j}{dz^j} F(z,x) \label{Leibniz}.
	\end{align}
	We shall now compute $	\lim_{z \to 1} \frac{d^j}{dz^j} F(z,x)$ for $x \in \mathbb{R}$, $0 \leq j \leq k$.   
	Differentiate both sides of \eqref{F(x)} $j$ times with respect to $z$ (which is justified since the series $\sum_{n=1}^{\infty}\varphi(z, nx)$ is uniformly convergent in $0<\textup{Re}(z)<2$; see \cite[p.~1159]{dixitijnt} for details) to get
	\begin{align}
		\frac{d^j}{dz^j}F(z,x)  &= \sum_{n=1}^{\infty} \left( \zeta^{(j)}(z, nx)- \frac{1}{2}(-1)^j \log^j(nx) (nx)^{-z} - \sum_{m=0}^{j} \binom{j}{m} \frac{(-1)^j m! \log^{j-m}(nx) (nx)^{1-z}}{(z-1)^{m+1}} \right) \notag \\
		&\hspace{1cm}-\frac{1}{2}\sum_{\ell=0}^{j} \binom{j}{\ell} (-1)^{j-\ell} \log^{j-\ell}(x) \zeta^{(\ell)}(z) x^{-z} - \sum_{\ell=0}^{j} \binom{j}{\ell} \frac{(-1)^{j-\ell} (j-\ell)! \zeta^{(\ell)}(z-1) }{x (z-1)^{j-\ell+1}} \notag \\
		&	=S_1(z,x) + S_2(z,x) \label{differentiated},
	\end{align}
	where
	\begin{align}
		&S_1(z,x):= \sum_{n=1}^{\infty} \left( \zeta^{(j)}(z, nx)- \frac{1}{2}(-1)^j \log^j(nx) (nx)^{-z} - \sum_{m=0}^{j} \binom{j}{m} \frac{(-1)^j m! \log^{j-m}(nx) (nx)^{1-z}}{(z-1)^{m+1}} \right), \label{S1} \\
		&S_2(z,x) := \sum_{\ell=0}^{j} \binom{j}{\ell} (-1)^{j-\ell+1} \left( \frac{ \log^{j-\ell}(x) \zeta^{(\ell)}(z) }{2x^z}+ \frac{ (j-\ell)! \zeta^{(\ell)}(z-1) }{x (z-1)^{j-\ell+1}} \right). \label{S2}
	\end{align}
	Our first goal is to explicitly evaluate $S_1(z,x)$ as $z \rightarrow 1 $. From \eqref{hurwitzlaurent} and \eqref{psigamma},
	\begin{align*}
		\zeta(z,a) = \frac{1}{z-1} + \sum_{\ell=0}^{\infty} (-1)^{\ell+1} \frac{\psi_\ell(a)}{\ell!} (z-1)^\ell
	\end{align*}
so that		
	\begin{align}
		\lim_{z \to 1}\left( \zeta^{(j)}(z, nx) + \frac{(-1)^{j+1} j!}{(z-1)^{j+1}  } \right) = (-1)^{j+1}  \psi_j(nx) \label{hurwitz_psi_j} .
	\end{align}	
	Now add and subtract $\frac{(-1)^{j+1} j!}{(z-1)^{j+1}}$ from the summand of $S_1(z,x)$ in \eqref{S1}, let $z \rightarrow 1$, and then use \eqref{hurwitz_psi_j} to have
	\begin{align}
		\lim_{z \to 1}	S_1(z,x) 
&=\sum_{n=1}^{\infty} \Bigg(\lim_{z \to 1}\left( \zeta^{(j)}(z, nx) + \frac{(-1)^{j+1} j!}{(z-1)^{j+1}  } \right)  +\lim_{z \to 1} \frac{1}{2}(-1)^{j+1} \log^j(nx) (nx)^{-z} \notag \\
		&\hspace{1cm}- \lim_{z \to 1} \left\{ \sum_{m=0}^{j} \binom{j}{m} \frac{(-1)^j m! \log^{j-m}(nx) (nx)^{1-z}}{(z-1)^{m+1}}+\frac{(-1)^{j+1} j!}{(z-1)^{j+1}  }\right\} \Bigg) \notag \\
		&= \sum_{n=1}^{\infty} \Bigg((-1)^{j+1}  \psi_j(nx)+    \frac{(-1)^{j+1} \log^j(nx) }{2nx} - \lim_{z \to 1} \frac{S_3(z, x)}{(z-1)^{j+1}} \Bigg) \label{firstlimitS1} ,
	\end{align}
	where 
	\begin{align}
		S_3(z, x)=  (-1)^{j+1} j! +\sum_{m=0}^{j} \binom{j}{m} (-1)^j m! \log^{j-m}(nx) (nx)^{1-z} (z-1)^{j-m} \label{S3(z,x)} .
	\end{align}
	Clearly, $S_3(1, x)=0$ since $j\geq1$. We shall now show that $S_3(z, x)$ has zero of order $j+1$ at $z=1$.
	For $1 \leq i \leq j$, differentiate  both sides of  \eqref{S3(z,x)} $i$ times with respect to $z$, and then let $z \rightarrow 1$ so as to have
	\begin{align*}
		\lim_{z \to 1} 	\frac{d^i}{dz^i} S_3(z,x)=  \sum_{m=0}^{j} \binom{j}{m} (-1)^j m! \log^{j-m}(nx) \sum_{r=0}^{i} \binom{i}{r} \lim_{z \to 1}\left\{  \frac{d^{i-r}}{dz^{i-r}} (nx)^{1-z} \frac{d^r}{dz^r}(z-1)^{j-m}\right\}.
	\end{align*}
	Here we note that, if $ j- m< r$, then $\frac{d^{r}}{dz^{r}} (z-1)^{j-m} =0 $. Also, if $ j-m>r$, then $ 	\lim_{z \to 1}  \frac{d^{r}}{dz^{r}}(z-1)^{j-m} =0 $. Hence the inner sum survives only for $ r=j-m$. Also, this is possible only when $ j-m \leq i$, i.e. $ m \geq j-i $. Therefore,
	\begin{align*}
		\lim_{z \to 1} 	\frac{d^i}{dz^i} S_3(z,x)
		&= \sum_{m=j-i}^{j} \binom{j}{m}  (-1)^j m! \log^{j-m}(nx) \binom{i}{j-m}  \log^{i-j+m}(nx) (-1)^{i-j+m} (j-m)!
		\\
		&= j! \log^i(nx) \sum_{m=j-i}^{j} (-1)^{i+m} \binom{i}{j-m} \\
		&= (-1)^j j! \log^i(nx) \sum_{m=0}^{i} (-1)^m \binom{i}{i-m}\\
		&=0,
	\end{align*}
	where, in the penultimate step we replaced $m$ by $m+j-i$ and then used the binomial theorem. Therefore, using L'Hospital's rule and observing that only the $r=m+1$ term survives in the inner finite sum in the second step below, we find that
	\begin{align}
		\lim_{z \to 1} \frac{S_3(z,x)}{(z-1)^{j+1}} 
		&= \lim_{z \to 1} \frac{  \frac{d^{j+1}}{dz^{j+1}} S_3(z,x)}{ \frac{d^{j+1}}{dz^{j+1}} (z-1)^{j+1}}	\notag \\
		&= \frac{1}{(j+1)!}\lim_{z \to 1} \sum_{m=0}^{j} \binom{j}{m} (-1)^j m! \log^{j-m}(nx) \sum_{r=0}^{j+1} \binom{j+1}{r} \frac{d^r}{dz^r} (nx)^{1-z} \frac{d^{j+1-r}}{dz^{j+1-r}}(z-1)^{j-m} \notag \\
		&= \frac{1}{(j+1)!} \sum_{m=0}^{j} \binom{j}{m}  (-1)^j m! \log^{j-m}(nx) \binom{j+1}{m+1}  \log^{m+1}(nx) (-1)^{m+1} (j-m)! \notag \\
		&=  \frac{(-1)^{j+1} j! \log^{j+1}(nx)}{(j+1)!}  \sum_{m=0}^{j}  (-1)^{m}   \binom{j+1}{m+1} \notag \\
		&= \frac{(-1)^{j+1} j! \log^{j+1}(nx)}{(j+1)!} \label{limitS3} , 
	\end{align}
	where in the last step we used the binomial theorem. From \eqref{firstlimitS1} and \eqref{limitS3}, we thus have
	\begin{align}
		\lim_{z \to 1}	S_1(z,x) = (-1)^{j+1} \sum_{n=1}^{\infty} \Bigg\{ \psi_j(nx)+    \frac{ \log^j(nx) }{2nx} -  \frac{ \log^{j+1}(nx)}{j+1} \Bigg\} \label{secondlimitS1} .
	\end{align}
	Next, we compute $\lim_{z \to 1}	S_2(z,x)$. Add and subtract $\binom{j}{\ell}\frac{(-1)^{j+1}\ell!\log^{j-\ell}(x)}{2x^z(z-1)^{\ell+1}}$ from the summand of  $S_2(z,x)$ in \eqref{S2} to get
	\begin{align}
		S_2(z,x)  
		= \sum_{\ell=0}^{j} \binom{j}{\ell} (-1)^{j-\ell+1}  \frac{ \log^{j-\ell}(x) }{2x^z}\left\{ \zeta^{(\ell)}(z) - \frac{(-1)^\ell \ell!}{(z-1)^{\ell+1}}\right\}-\frac{\left(S_4(z,x)+S_5(z,x)\right)}{2x^z(z-1)^{j+1}} \label{S2simplified} ,
	\end{align}
	where, with $g_\ell(z) := (z-1)^\ell x^{z-1} \zeta^{(\ell)}(z-1), 0 \leq \ell \leq j$, we define
	\begin{align}
		S_4(z,x)&:= 2\sum_{\ell=0}^{j} (-1)^{j-\ell}  \binom{j}{\ell}   (j-\ell)! g_\ell(z), \label{S4} \\
		S_5(z,x)&:= \sum_{\ell=0}^{j} (-1)^j \binom{j}{\ell} \ell! \log^{j-\ell}(x) (z-1)^{j-\ell}. \label{S5}
	\end{align}
	Now we show that $S_4(z,x)+S_5(z,x)$ has a zero of order $j+1$ at $z=1$. 
	Note that $g_\ell(z) =  h_1(z) h_2(z)$, where $h_1(z) = (z-1)^\ell, h_2(z)= x^{z-1} \zeta^{(\ell)}(z-1). $ Note that $ \lim_{z \to 1} \frac{d^s}{dz^s}h_1(z) =0 $ unless $s=\ell$. Also,  $ \lim_{z \to 1} \frac{d^\ell}{dz^\ell}h_1(z) =\ell!. $ Therefore,
	\begin{align}
		\lim_{z \to 1} \frac{d^i}{dz^i}g_\ell(z) &= \lim_{z \to 1} \sum_{s=0}^{i}  \binom{i}{s} \frac{d^s}{dz^s}h_1(z) \frac{d^{i-s}}{dz^{i-s}} h_2(z) \notag \\
		&=    \begin{cases}
			0 & \text{if}~ i<\ell, \\
			\binom{i}{\ell} \ell! \lim_{z \to 1} \frac{d^{i-\ell}}{dz^{i-\ell}} h_2(z)  & \text{if} ~ i \geq \ell.
		\end{cases} \label{derivativeg_l}
	\end{align}
	For $i \geq \ell$, 
	\begin{align}
		\lim_{z \to 1} \frac{d^{i-\ell}}{dz^{i-\ell}}h_2(z)&=
		\lim_{z \to 1} \sum_{r=0}^{i-\ell} \binom{i-\ell}{r} \log^r(x) x^{z-1} \zeta^{(i-r)}(z-1) \notag \\
		&= \sum_{r=0}^{i-\ell} \binom{i-\ell}{r} \log^r(x)  \zeta^{(i-r)}(0) \label{derivativeh_2} .
	\end{align}
	Substituting \eqref{derivativeh_2} in \eqref{derivativeg_l}, we get
	\begin{align}
		\lim_{z \to 1} \frac{d^i}{dz^i}g_\ell(z) =    \begin{cases}
			0 & \text{if}~ i<\ell,  \\
			\binom{i}{\ell} \ell! \sum_{r=0}^{i-\ell} \binom{i-\ell}{r} \log^r(x)  \zeta^{(i-r)}(0) & \text{if}~ i \geq \ell.
		\end{cases}  \label{finallimitg_l}
	\end{align}
	Substitute \eqref{finallimitg_l} in \eqref{S4} and simplify to obtain
	\begin{align}
		\lim_{z \to 1} 	\frac{d^i}{dz^i} S_4(z,x)
		=  2 j !\sum_{\ell=0}^{i} (-1)^{j-\ell}   \binom{i}{\ell} \sum_{r=0}^{i-\ell} \binom{i-\ell}{r} \log^r(x)  \zeta^{(i-r)}(0).\label{claimlimitS4} 
	\end{align}
	Similarly, it can be seen that
	\begin{align}
		\lim_{z \to 1} 	\frac{d^i}{dz^i} S_5(z,x)
		=  j! (-1)^j \log^i(x) \label{limitS5} .
	\end{align}
	Hence from \eqref{claimlimitS4} and \eqref{limitS5}, for $0 \leq i \leq j$,
	\begin{align*}
		&\lim_{z \to 1} 	\frac{d^i}{dz^i} (S_4(z,x)+S_5(z,x))  \\
		&= j !\Big\{2 (-1)^{j-i}\sum_{\ell=0}^{i} (-1)^{i-\ell}   \binom{i}{\ell} \sum_{r=0}^{i-\ell} \binom{i-\ell}{r} \log^r(x)  \zeta^{(i-r)}(0)+  (-1)^j \log^i(x) \Big\}\\
		&= j !\Big\{2 (-1)^{j-i}\sum_{r=0}^{i} \log^r(x)  \zeta^{(i-r)}(0)  \sum_{\ell=0}^{i-r} (-1)^{i-\ell}   \binom{i}{\ell} \binom{i-\ell}{r} +  (-1)^j \log^i(x) \Big\}\\
		&=j !\Big\{2 (-1)^{j-i}\sum_{r=0}^{i} (-1)^{r}  \binom{i}{r}  \log^r(x)  \zeta^{(i-r)}(0)  \sum_{\ell=0}^{i-r} (-1)^{i-r-\ell} \binom{i-r}{i-r-\ell}  +  (-1)^j \log^i(x) \Big\}\\
		&= j! \Big\{2 (-1)^{j} \zeta(0)   \log^i(x) +  (-1)^j \log^i(x) \Big\} \\
		&=0,
	\end{align*}
where, in the penultimate step, we used the binomial theorem to see that the inner sum vanishes unless $r=i$, and in which case is equal to $1$. Also, in the last step we have used the fact that $\zeta(0) = -\frac{1}{2}$. Hence $S_4(z,x)+S_5(z,x)$ has a zero of order $j+1$ at $z=1$. Therefore by L'Hospital's rule,
	\begin{align*}
		\lim_{z \to 1} \frac{S_4(z,x)+S_5(z,x)}{2x^z(z-1)^{j+1}} =\frac{1}{2x}\lim_{z \to 1} \frac{ 	\frac{d^{j+1}}{dz^{j+1}} (S_4(z,x)+S_5(z,x))}{\frac{d^{j+1}}{dz^{j+1}} (z-1)^{j+1}}.
	\end{align*}
	Using \eqref{S4} and \eqref{S5}, we see that the above limit equals
	\begin{align*}
		 \frac{\lim_{z \to 1}\left[2\sum_{\ell=0}^{j} (-1)^{j-\ell}  \binom{j}{\ell}   (j-\ell)! 	\frac{d^{j+1}}{dz^{j+1}} g_\ell(z) +   \sum_{\ell=0}^{j} (-1)^j \binom{j}{\ell} \ell! \log^{j-\ell}(x) \frac{d^{j+1}}{dz^{j+1}}  (z-1)^{j-\ell}\right]}{2x(j+1)!}.
	\end{align*}
	Observe that for all $0 \leq \ell \leq j$, $\frac{d^{j+1}}{dz^{j+1}} (z-1)^{j-\ell}=0$, and hence, from \eqref{finallimitg_l},
	\begin{align*}
		\lim_{z \to 1} \frac{d^{j+1}}{dz^{j+1}}g_\ell(z) =   
		\binom{j+1}{\ell} \ell! \sum_{r=0}^{j+1-\ell} \binom{j+1-\ell}{r} \log^r(x)  \zeta^{(j+1-r)}(0).
	\end{align*}
	Therefore,
	\begin{align}
		\lim_{z \to 1} & \frac{S_4(z,x)+S_5(z,x)}{2x^z(z-1)^{j+1}} = \frac{1}{x(j+1)} \sum_{\ell=0}^{j} (-1)^{j-\ell} \binom{j+1}{\ell}  \sum_{r=0}^{j+1-\ell} \binom{j+1-\ell}{r} \log^r(x)  \zeta^{(j+1-r)}(0) \notag \\
		&= \frac{1}{x(j+1)}\Big[\sum_{\ell=0}^{j} (-1)^{j-\ell} \binom{j+1}{\ell} \zeta^{(j+1)}(0) + \sum_{\ell=0}^{j} (-1)^{j-\ell} \binom{j+1}{\ell}  \sum_{r=1}^{j+1-\ell} \binom{j+1-\ell}{r} \log^r(x)  \zeta^{(j+1-r)}(0)\Big] \notag \\
		&= \frac{1}{x(j+1)}\Big[ \zeta^{(j+1)}(0)\Big\{1-\sum_{\ell=0}^{j+1} (-1)^{j+1-\ell} \binom{j+1}{\ell}\Big\} \notag \\
		&\hspace{2cm}+ \sum_{\ell=0}^{j} (-1)^{j-\ell} \binom{j+1}{\ell}  \sum_{r=0}^{j-\ell} \binom{j+1-\ell}{r+1} \log^{r+1}(x)  \zeta^{(j-r)}(0)\Big] \notag \\
		&=\frac{1}{x(j+1)}\Big[ \zeta^{(j+1)}(0) + \sum_{r=0}^{j} \log^{r+1}(x)  \zeta^{(j-r)}(0) \sum_{\ell=0}^{j-r} (-1)^{j-\ell} \binom{j+1}{\ell}  \binom{j+1-\ell}{r+1} \Big] \notag \\
		&=\frac{1}{x(j+1)}\Big[ \zeta^{(j+1)}(0) + \sum_{r=0}^{j} (-1)^r \binom{j+1}{r+1} \log^{r+1}(x)  \zeta^{(j-r)}(0)  \sum_{\ell=0}^{j-r} (-1)^{j-r-\ell} \binom{j-r}{j-r-\ell}  \Big] \notag \\
		&=\frac{1}{x(j+1)}\Big[ \zeta^{(j+1)}(0) + (-1)^j \log^{j+1}(x)  \zeta(0)    \Big] \notag \\
		&=\frac{2 \zeta^{(j+1)}(0)+(-1)^{j+1} \log^{j+1}(x)}{2x(j+1)}. \label{limitS4S5}
	\end{align}
	Let $z \rightarrow 1$ in \eqref{S2simplified} and employ \eqref{limitS4S5} to get
	\begin{align}
		\lim_{z \to 1}	S_2(z,x) 
		&= \sum_{\ell=0}^{j} \binom{j}{\ell} (-1)^{j+1}  \frac{ \log^{j-\ell}(x) \gamma_\ell }{2x}  - \frac{2 \zeta^{(j+1)}(0)+(-1)^{j+1} \log^{j+1}(x)}{2x(j+1)}	\label{limitS2},
	\end{align}
	since 
	\begin{align*}
		\lim_{z \to 1}\left\{ \zeta^{(\ell)}(z) - \frac{(-1)^\ell \ell!}{(z-1)^{\ell+1}}\right\}= (-1)^\ell \gamma_\ell.
	\end{align*}
	From \eqref{differentiated}, \eqref{secondlimitS1} and \eqref{limitS2},
	\begin{align}
		\lim_{z \to 1} \frac{d^j}{dz^j} F(z,x)
		&= (-1)^{j+1} \left[ \sum_{n=1}^{\infty} \Bigg\{ \psi_j(nx)+    \frac{ \log^j(nx) }{2nx} -  \frac{ \log^{j+1}(nx)}{(j+1)} \Bigg\} \right.\notag \\
		&\hspace{5mm} +\left.\sum_{\ell=0}^{j} \binom{j}{\ell}  \frac{ \log^{j-\ell}(x) \gamma_\ell }{2x}  - \frac{(-1)^{j+1} 2 \zeta^{(j+1)}(0)+ \log^{j+1}(x)}{2x(j+1)}	 \right] \label{lastlimit} .
	\end{align}
	Finally, substituting \eqref{lastlimit} in \eqref{Leibniz} and using \eqref{diffIJNT} we arrive at \eqref{Ramanujan Identity} for $\alpha, \beta>0$. Since both sides of \eqref{Ramanujan Identity} are analytic in Re$(\alpha)>0$, Re$(\beta)>0$, the result holds in the extended region as well.  
\end{proof}

When $k=1$, in addition to the modular relation in the theorem proved above, we also derive an analogue of the integral containing the $\Xi$-function in \eqref{w1.26}.
 \begin{proof}[Theorem \textup{\ref{ramanujan analogue psi_1 infinite}}][]
	Define 
	\begin{align}\label{j z alpha}
		J(z, \alpha)	:= \frac{8 (4 \pi)^{\frac{z-4}{2}}} {\Gamma(z)} I(z, \alpha),
	\end{align}
	where 
	\begin{equation}\label{i zee}
	I(z, \alpha):= \int_{0}^{\infty} \omega(z,t)\cos\left(\tf{1}{2}t\log\a\right)\, dt,
\end{equation}
 and  $\omega(z, t)$ is defined in \eqref{omdef}. Differentiate extreme sides of \eqref{mainneq2} with respect to $z$, then let $z\to 1$ and use \eqref{F(x)} to get
	\begin{align}\label{extra}
		\sqrt{\alpha} \sum_{j=0}^{1} \binom{1}{j} \frac{\log^{1-j}(x)}{2^{1-j}} 	\lim_{z \to 1} \frac{d^j}{dz^j} \left(F(z, \alpha) \right)
		=\lim_{z\to 1}\frac{d}{dz}J(z, \alpha).
		\end{align}
		Substituting \eqref{lastlimit} with $x=\alpha$ on the left-hand side of the above equation, we see that
		\begin{align}\label{lhs extra}
			&\sqrt{\alpha} \sum_{j=0}^{1} \binom{1}{j} \frac{\log^{1-j}(x)}{2^{1-j}} 	\lim_{z \to 1} \frac{d^j}{dz^j} \left(F(z, \alpha) \right)\notag\\
			&= \sqrt{\alpha} \sum_{j=0}^{1} \frac{\log^{1-j}(\alpha)}{2^{1-j}} (-1)^{j+1} \left[ \sum_{n=1}^{\infty} \Bigg\{ \psi_j(n\alpha)+    \frac{ \log^j(n\alpha) }{2n\alpha} -  \frac{ \log^{j+1}(n\alpha)}{(j+1)} \Bigg\} \right. \notag \\
			& \left. +\sum_{\ell=0}^{j} \binom{j}{\ell}  \frac{ \log^{j-\ell}(\alpha) \gamma_\ell }{2\alpha}  - \frac{(-1)^{j+1} 2 \zeta^{(j+1)}(0)+ \log^{j+1}(\alpha)}{2\alpha(j+1)}	 \right]\notag\\
			&=\mathscr{F}_1(\alpha),
		\end{align}
		where $\mathscr{F}_1(\alpha)$ is defined in \eqref{ef1}, and where we used the representation \cite[p.~226]{apostol1}
		\begin{equation}\label{zeta''0} \zeta''(0)=-\frac{1}{2}\log^{2}(2\pi)-\frac{\pi^2}{24}+\frac{1}{2}\gamma^2+\gamma_1.
		\end{equation}
	Also,
		\begin{align}
		\lim_{z\to 1}\frac{d}{dz}J(z, \alpha)&= \lim_{z \to 1} \left[ \frac{8 (4 \pi)^{\frac{z-4}{2}}} {\Gamma(z)} I'(z, \alpha) + \frac{\Gamma(z)8 (4 \pi)^{\frac{z-4}{2}} \log(4\pi) \frac{1}{2}-8 (4 \pi)^{\frac{z-4}{2}}\Gamma'(z)}{\Gamma^{2}(z)} I(z, \alpha) \right]\notag\\
		&=\pi^{\frac{-3}{2}}\left\{ \lim_{z \to 1}  I'(z, \alpha) + \left(  \frac{1}{2} \log(4\pi)+\gamma \right) \lim_{z \to 1}  I(z, \alpha) \right\}. \label{II'}
	\end{align}
	Take logarithmic derivative on both sides of \eqref{omdef}  with respect to $z$ to get
	\begin{align*}
		\frac{\omega'(z,t)}{\omega(z,t)}= \frac{1}{4} \psi \left( \frac{z-2+it}{4} \right) + \frac{1}{4} \psi \left( \frac{z-2-it}{4} \right) + \frac{i}{2} \frac{\Xi'\left( \frac{t+iz-i}{2} \right)}{\Xi \left( \frac{t+iz-i}{2} \right)} - \frac{i}{2} \frac{\Xi'\left( \frac{t-iz+i}{2} \right)}{\Xi \left( \frac{t-iz+i}{2} \right)} - \frac{2z}{z^2+t^2}.		
	\end{align*}
	Now let $z \to 1$ so that
	\begin{align*}
		\omega'(1,t)= \omega(1,t) \left\{ \frac{1}{4} \psi \left( \frac{-1+it}{4} \right) + \frac{1}{4} \psi \left( \frac{-1-it}{4} \right) - \frac{2}{1+t^2} \right\}.		
	\end{align*}
	Observe that from \eqref{i zee},
	\begin{align*}
		\lim_{z \to 1} I'(z, \alpha) &= \int_{0}^{\infty} \lim_{z \to 1} \omega'(z,t)\cos\left(\tf{1}{2}t\log\a\right) \, dt\\
		&=\int_{0}^{\infty} \omega(1,t) \left\{ \frac{1}{4} \psi \left( \frac{-1+it}{4} \right) + \frac{1}{4} \psi \left( \frac{-1-it}{4} \right) - \frac{2}{1+t^2} \right\} \cos\left(\tf{1}{2}t\log\a\right)\, dt,
	\end{align*}
	where the interchange of the order of limit and integration is permissible because of the uniform convergence of the integral in $0<\textup{Re}(z)<2$, which, in turn, follows from \eqref{strivert} and elementary bounds of the zeta function. Substitute the above equation in \eqref{II'}, simplify and use \eqref{mathscr i alpha} to arrive at 
	\begin{align}\label{used in asymptote}
	\lim_{z\to 1}\frac{d}{dz}J(z, \alpha)=\mathscr{I}(\alpha).
	\end{align}
	Thus, from \eqref{extra}, \eqref{lhs extra} and \eqref{used in asymptote}, we see that 	$\mathscr{F}_1(\alpha)=\mathscr{I}(\alpha)$. Noting that $\mathscr{I}(\alpha)$ is invariant upon replacing $\alpha$ by $\beta$, where $\alpha\beta=1$, we get the remaining equalities in \eqref{k=1}, which completes the proof.
\end{proof}	
	
	\subsection{Proof of the asymptotic expansion of the $\Xi$-function integral}

 \begin{proof}[Theorem \textup{\ref{asymptotic_xit}}][]
From \cite[Equation (6.2)]{drzacta}, for $-1 < \Re z <1$ and Re$(\alpha)>0$\footnote{In \cite{drzacta}, this result is stated for $\alpha>0$, however, by analytic continuation, it can be easily seen to be true for Re$(\alpha)>0$.}, 
	\begin{align}
		\alpha^{(z+1)/2} \int_{0}^{\infty} e^{-2 \pi \alpha x} x^{z/2} \left( \Omega(x,z) - \frac{1}{2\pi} \zeta(z)x^{z/2-1} \right) \, dx
		= \frac{1} {2 \pi^{(z+5)/2}}I(z+1, \alpha), \label{omegaXi}
	\end{align}	
	where, $$\Omega(x,z)= 2 \sum_{n=1}^{\infty} \sigma_{-z}(n) n^{z/2} \left(e^{\pi iz/4}K_z(4\pi e^{\pi i/4} \sqrt{nx})+ e^{-\pi iz/4}K_z(4\pi e^{-\pi i/4} \sqrt{nx})\right),$$
	with $K_z(x)$ denoting the modified Bessel function of the second kind and $\sigma_{-z}(n)= \sum_{d|n} d^{-z}$. Therefore, from \eqref{j z alpha} and \eqref{omegaXi},
\begin{align}
\alpha^{(z+1)/2} \int_{0}^{\infty} e^{-2 \pi \alpha x} x^{z/2} \left( \Omega(x,z) - \frac{1}{2\pi} \zeta(z)x^{z/2-1} \right) \, dx= \frac{\Gamma(z+1)}{(2\pi)^{z+1}} J(z+1, \alpha). \label{omegaXi2}
	\end{align}
	Differentiate both sides of \eqref{omegaXi2} with respect to $z$ and then let $z\to 0$ so as to have
	\begin{align}
		&\lim_{z\to 0}\frac{d}{dz} \left[ \alpha^{(z+1)/2} \int_{0}^{\infty} e^{-2 \pi \alpha x} x^{z/2} \left( \Omega(x,z) - \frac{1}{2\pi} \zeta(z)x^{z/2-1} \right) \, dx \right] \notag \\
		&=\frac{1}{2\pi} \lim_{z\to 0}\frac{d}{dz}J(z+1, \alpha)
		- \frac{\gamma+\log(2\pi)}{(2\pi)} J( 1, \alpha),\label{omegaXidiff}
	\end{align}
	since $\Gamma'(1)=-\gamma$.
	Our next task is to obtain the asymptotic expansion of the left-hand side as $\alpha \to \infty$. 
	Note that from \cite[Proposition 6.1]{dixitmoll}, for $-1<\textup{Re}(z)<1$,
	\begin{align}\label{omega x z}
		\Omega(x,z) - \frac{1}{2\pi} \zeta(z)x^{z/2-1}  =\frac{-\Gamma(z)\zeta(z)}{(2\pi \sqrt{x})^z} -\frac{x^{z/2}}{2}\zeta(z+1) + \frac{x^{z/2+1}}{\pi} \sum_{n=1}^{\infty} \frac{\sigma_{-z}(n)}{n^2+x^2},
	\end{align}
	which, upon letting $z\to0$ gives
	\begin{equation}\label{omega x 0}
		\Omega(x,0) + \frac{1}{4\pi x} = -\gamma -\frac{1}{2}\log(x) +\frac{x}{\pi} \sum_{n=1}^{\infty} \frac{d(n)}{n^2+x^2}.
	\end{equation}
Using \eqref{omega x z}, we have
	\begin{align*}
		&\frac{d}{dz} \left[ (\alpha x)^{z/2} \left( \Omega(x,z) - \frac{1}{2\pi} \zeta(z)x^{z/2-1} \right) \right] \\
		=&  \frac{(\alpha x)^{z/2} \log(\alpha x)}{2}  \left( \Omega(x,z) - \frac{1}{2\pi} \zeta(z)x^{z/2-1} \right) + (\alpha x)^{z/2} \frac{d}{dz} \left( \frac{-\Gamma(z)\zeta(z)}{(2\pi \sqrt{x})^z} -\frac{x^{z/2}}{2}\zeta(z+1) + \frac{x^{z/2+1}}{\pi} \sum_{n=1}^{\infty} \frac{\sigma_{-z}(n)}{n^2+x^2} \right)\\
		=& \frac{(\alpha x)^{z/2} \log(\alpha x)}{2}  \left( \Omega(x,z) - \frac{1}{2\pi} \zeta(z)x^{z/2-1} \right) \\
		&+ (\alpha x)^{z/2} \left( -\frac{(\Gamma'(z)\zeta(z)+\Gamma(z)\zeta'(z)- \log(2\pi \sqrt{x}) \Gamma(z)\zeta(z))}{(2\pi \sqrt{x})^{z}}-\frac{x^{z/2}}{4}\log(x)\zeta(z+1)-\frac{x^{z/2}}{2}\zeta'(z+1)  \right. \\
		&\left.  + \frac{x^{z/2+1}}{2\pi}\log(x) \sum_{n=1}^{\infty} \frac{\sigma_{-z}(n)}{n^2+x^2} -  \frac{x^{z/2+1}}{\pi} \sum_{d=1}^{\infty} \sum_{m=1}^{\infty} \frac{d^{-z}\log(d)}{m^2d^2+x^2} \right).
	\end{align*}
Since 
	\small\begin{align*}
	\lim_{z\to0}(\alpha x)^{z/2} \left( -\frac{(\Gamma'(z)\zeta(z)+\Gamma(z)\zeta'(z)- \log(2\pi \sqrt{x}) \Gamma(z)\zeta(z))}{(2\pi \sqrt{x})^{z}}-\frac{x^{z/2}}{4}\log(x)\zeta(z+1)-\frac{x^{z/2}}{2}\zeta'(z+1)\right) =\frac{\pi^2}{16},
\end{align*}
\normalsize
which can be proved using the well-known power series/Laurent series expansions about $z=0$ of the functions involved, we obtain
	\begin{align}
	&\lim_{z\to0}\frac{d}{dz} \left[ (\alpha x)^{z/2} \left( \Omega(x,z) - \frac{1}{2\pi} \zeta(z)x^{z/2-1} \right) \right] \notag\\	=& \frac{\log(\alpha x)}{2}  \left( \Omega(x,0) + \frac{1}{4\pi x} \right) + \frac{\pi^2}{16} + \frac{1}{\pi} \left( \frac{x\log(x)}{2} \sum_{n=1}^{\infty} \frac{d(n)}{n^2+x^2} - \frac{x}{2} \sum_{n=1}^{\infty} \frac{d(n)\log(n)}{n^2+x^2} \right), \label{omegaderivative}
	\end{align}
where we used the fact $\sum_{d|n}\log(d)=\log(\prod_{d|n}d)=\frac{1}{2}d(n)\log(n)$ using the elementary result $\prod_{d|n}d=n^{d(n)/2}$ given, for example, in \cite[Exercise 10, p.~47]{Apostol}. 
Since the integral on the left-hand side of \eqref{omegaXidiff} converges uniformly in $-1<\textup{Re}(z)<1$ as can be seen from the bounds established for the integrand in \cite[Equations (9.1.2), (9.1.3)]{dkRiMS1}, from \eqref{omegaXidiff}  and \eqref{omegaderivative}, we have
	\begin{align}
		&\sqrt{\alpha} \lim_{z \to 0} \frac{d}{dz} \int_{0}^{\infty} e^{-2 \pi \alpha x} (\alpha x)^{z/2} \left( \Omega(x,z) - \frac{1}{2\pi} \zeta(z)x^{z/2-1} \right) \, dx\notag \\
		&=  \frac{\sqrt{\alpha} \log(\alpha)}{2} \int_{0}^{\infty} e^{-2 \pi \alpha x} \left( \Omega(x,0) + \frac{1}{4\pi x} \right) \, dx + \frac{\sqrt{\alpha}}{2} \int_{0}^{\infty} e^{-2 \pi \alpha x} \log(x) \left( \Omega(x,0) + \frac{1}{4\pi x} \right) \, dx \notag \\
		&+ \frac{\pi^2 \sqrt{ \alpha}  }{16} \int_{0}^{\infty} e^{-2 \pi \alpha x} \, dx  + \frac{\sqrt{\alpha}}{2\pi} \int_{0}^{\infty} xe^{-2 \pi \alpha x} \left( \log(x) \sum_{n=1}^{\infty} \frac{d(n)}{n^2+x^2} -  \sum_{n=1}^{\infty} \frac{d(n)\log(n)}{n^2+x^2} \right) \, dx\notag\\
		&=:H_1(\alpha)+H_2(\alpha)+H_3(\alpha)+H_4(\alpha). \label{LHSasy}
	\end{align}
	Equations (6.3) and (6.4) of \cite{drzacta}, which are obtained there for $\alpha>0$, actually hold for Re$(\alpha)>0$ (as can be seen from the sentence following Theorem \ref{watlem}). Thus together they imply that as $\alpha\to\infty$,
	\begin{align}
		H_1(\alpha) \sim \frac{\sqrt{\alpha} \log(\alpha)}{2} \left( \frac{-\gamma +\log(2\pi \alpha)}{4\pi \alpha} + \sum_{m=0}^{\infty} \frac{(-1)^m}{\pi (2\pi \alpha)^{2m+2}}\Gamma(2m+2) \zeta^2(2m+2) \right). \label{I1}
	\end{align}
	We now find the asymptotic expansion of $H_2(\alpha)$ as $\alpha\to\infty$. Using \eqref{omega x 0}, we have
	\begin{align}\label{H2 decomposition}
	H_2(\alpha)
		&=  \frac{\sqrt{\alpha}}{2} \int_{0}^{\infty} e^{-2 \pi \alpha x} \log(x) \left( -\gamma -\frac{1}{2}\log(x) \right) \, dx + \frac{\sqrt{\alpha} }{2} \int_{0}^{\infty} e^{-2 \pi \alpha x} \log(x) \frac{x}{\pi} \sum_{n=1}^{\infty} \frac{d(n)}{n^2+x^2}  \, dx\notag\\
		&=:H_{21}(\alpha)+H_{22}(\alpha).
	\end{align}
Using \cite[p.~571, formula \textbf{4.331.1}; p.~572, formula \textbf{4.335.1}]{grn},$H_{21}(\alpha)$ simplifies to
	\begin{align}
		H_{21}(\alpha)= \frac{1}{48 \pi \sqrt{\alpha}} \left( 6\gamma^2 - \pi^2 -6 \log^2(2) - 6\log(\pi)\log(4\pi) - 6\log(\alpha)\log(4\pi^2\alpha) \right).\label{I2first}
	\end{align}
	From \cite[p.~30]{drzacta}, as $x\to0$,
	\begin{align*}
		\frac{x}{\pi} \sum_{n=1}^{\infty} \frac{d(n)}{n^2+x^2} \sim \frac{1}{\pi}\sum_{m=0}^{\infty} (-1)^m \zeta^2(2m+2) x^{2m+1}.
	\end{align*}
	We are now ready to employ the generalization of Watson's lemma given in Theorem \ref{watsongen} with $p_m(x)= x$, $\lambda_m=2m+2$, $c_n= \frac{1}{\pi}(-1)^{m} \zeta^2(2m+2)$, $m\geq0$, which leads to
	\begin{align}
	H_{22}(\alpha) &\sim \frac{\sqrt{\alpha}}{2\pi} \sum_{m=0}^{\infty} (-1)^{m} \zeta^2(2m+2) \frac{d}{d\lambda_m} \left( \frac{\Gamma(\lambda_m)}{(2\pi\alpha)^{\lambda_m}} \right) \notag \\
		=& \frac{\sqrt{\alpha}}{2\pi} \sum_{m=0}^{\infty} \frac{(-1)^{m} \Gamma(2m+2) \zeta^2(2m+2)}{(2\pi\alpha)^{2m+2}} \left( \psi(2m+2) - \log(2\pi\alpha) \right). \label{I2second}
	\end{align}
	From \eqref{H2 decomposition}, \eqref{I2first} and \eqref{I2second},
	\begin{align}
		H_2(\alpha) &= \frac{1}{48 \pi \sqrt{\alpha}} \left( 6\gamma^2 - \pi^2 -6 \log^2(2) - 6\log(\pi)\log(4\pi) - 6\log(\alpha)\log(4\pi^2\alpha) \right) \notag \\
		&\quad+\frac{\sqrt{\alpha}}{2\pi} \sum_{m=0}^{\infty} \frac{(-1)^{m} \Gamma(2m+2) \zeta^2(2m+2)}{(2\pi\alpha)^{2m+2}} \left( \psi(2m+2) - \log(2\pi\alpha) \right). \label{I2}
	\end{align}	
	Also, we can directly evaluate $H_3(\alpha)$ to get
	\begin{align}
	H_3(\alpha) = \frac{\pi}{32 \sqrt{\alpha}}. \label{I3}
	\end{align}
	It only remains to find the asymptotic expansion of $H_4(\alpha)$ as $\alpha\to\infty$ for which we employ Watson's lemma, that is, Theorem \ref{watlem}. From the respective definitions of $H_4(\alpha)$ and $H_{22}(\alpha)$ from \eqref{LHSasy} and  \eqref{H2 decomposition}, we observe that
	\begin{equation}\label{H4 decomposition}
		H_4(\alpha)=H_{22}(\alpha)-\frac{\sqrt{\alpha}}{2\pi} \int_{0}^{\infty} xe^{-2 \pi \alpha x} \sum_{n=1}^{\infty} \frac{d(n)\log(n)}{n^2+x^2}  \, dx.
	\end{equation}
	As $x \to 0$,
	\begin{align*}
		x\sum_{n=1}^{\infty} \frac{d(n)\log(n)}{n^2+x^2} 
		=x \sum_{n=1}^{\infty} \frac{d(n)\log(n)}{n^2} \sum_{\ell=0}^{\infty} \left( \frac{-x^2}{n^2} \right)^\ell 
		=  2\sum_{\ell=0}^{\infty} (-1)^{\ell+1} \zeta(2\ell+2) \zeta'(2\ell+2) x^{2\ell+1},
	\end{align*}
	since, for $\textup{Re}(w)>1$, we have $\sum_{n=1}^{\infty}d(n)\log(n)n^{-w}=-2\zeta(w)\zeta'(w)$.
	Now use Watson's Lemma, that is, Theorem \ref{watlem} with $\mu=1/2$, $\lambda=1$, $y=2\pi\alpha$ and $a_s=2(-1)^{s+1} \zeta(2s+2) \zeta'(2s+2) $ to get, as $\alpha \to \infty$,
	\begin{align}
		\frac{\sqrt{\alpha}}{2\pi} \int_{0}^{\infty} xe^{-2 \pi \alpha x}  \sum_{n=1}^{\infty} \frac{d(n)\log(n)}{n^2+x^2}\, dx\sim \frac{\sqrt{\alpha}}{2\pi} \sum_{m=0}^{\infty} \frac{2 (-1)^{m+1}\Gamma(2m+2)\zeta(2m+2)\zeta'(2m+2)}{(2\pi\alpha)^{2m+2}}. \label{I4}
	\end{align}
	From \eqref{LHSasy}, \eqref{I1}, \eqref{I2second}, \eqref{I2}, \eqref{I3}, \eqref{H4 decomposition} and \eqref{I4}, we have, upon simplification,
	\begin{align}
		&\lim_{z \to 0}  \frac{d}{dz} \alpha ^{(z+1)/2}   \int_{0}^{\infty} e^{-2 \pi \alpha x} x^{z/2} \left( \Omega(x,z) - \frac{1}{2\pi} \zeta(z)x^{z/2-1} \right) \, dx \notag \\
		=& \frac{1}{8\pi\sqrt{\alpha}} \left( \gamma + \log(2\pi) \right) \left( \gamma- \log(2\pi\alpha) \right) +\frac{\pi}{96\sqrt{\alpha}} \notag \\ 
		+&\frac{\sqrt{\alpha}}{\pi} \sum_{m=0}^{\infty} \frac{(-1)^{m} \Gamma(2m+2) \zeta^2(2m+2)}{(2\pi\alpha)^{2m+2}} \left( \frac{\log(\alpha)}{2} + \psi(2m+2) - \log(2\pi\alpha)+ \frac{\zeta'(2m+2)}{\zeta(2m+2)} \right). \label{OmegaAsy}
	\end{align}
	Now from the definition of $J(z, \alpha)$ from \eqref{j z alpha} and \cite[Theorem 1.10]{drzacta}, we see that as $\alpha\to\infty$,
	\begin{align}
		J(1, \alpha) \sim \frac{-\gamma+\log(2\pi\alpha)}{2\sqrt{\alpha}}+ 2\sqrt{\alpha} \sum_{m=0}^{\infty} \frac{(-1)^m}{(2\pi\alpha)^{2m+2}} \Gamma(2m+2) \zeta^2(2m+2). \label{RHSknownasy}
	\end{align}
From \eqref{omegaXidiff}, \eqref{OmegaAsy} and \eqref{RHSknownasy},
	\begin{align*}
		\frac{1}{2\pi} \lim_{z\to 0}\frac{d}{dz}J(z+1, \alpha)
		&\sim -\frac{1}{8\pi\sqrt{\alpha}} \left( \gamma + \log(2\pi) \right) \left( \gamma- \log(2\pi\alpha) \right) +\frac{\pi}{96\sqrt{\alpha}} \\
		&\quad+ \frac{\sqrt{\alpha}}{\pi} \sum_{m=0}^{\infty} \frac{(-1)^{m} \Gamma(2m+2) \zeta^2(2m+2)}{(2\pi\alpha)^{2m+2}} \left(- \frac{1}{2}\log(\alpha) +\gamma+ \psi(2m+2) + \frac{\zeta'(2m+2)}{\zeta(2m+2)} \right).
	\end{align*}
	Finally, use \eqref{used in asymptote} in the above equation to complete the proof of \eqref{alpha infinity}.
\end{proof}
	
	 \begin{proof}[Corollary \textup{\ref{phi1 asymptotic}}][]
	 	We first prove \eqref{infinity alpha}. From \eqref{ef1} and \eqref{k=1},
	 	\begin{align}\label{neww}
	 	\sqrt{\alpha} \left\{ \sum_{n=1}^{\infty} \phi_1(n\alpha) + \frac{\log^2(2\pi)-(\gamma-\log(x))^2}{4\alpha} + \frac{\pi^2}{48\alpha} \right\} =\frac{\sqrt{\alpha} \log(\alpha)}{2} \left\{ \sum_{n=1}^{\infty} \phi(n\alpha) +\frac{\gamma - \log (2\pi\alpha)}{2\alpha} \right\}+\mathscr{I}(\alpha).
	 	\end{align}
	 Using 
	 	\eqref{w1.26}, \eqref{ramanujan asymptotic} and Theorem \ref{asymptotic_xit}, we arrive at \eqref{infinity alpha} upon simplification. 
	 	
	 	We next prove \eqref{alpha zero}. Recall from the sentence following \eqref{alpha infinity} that as $\alpha\to 0$,
	 		\begin{align}\label{zero alpha}
	 		\mathscr{I}(\alpha)\sim& -\frac{\sqrt{\alpha}}{4} \left( \gamma + \log(2\pi) \right) \left( \gamma- \log\left(\tfrac{2\pi}{\alpha}\right) \right) +\frac{\pi^2\sqrt{\alpha}}{48}\notag \\
	 		&+ \frac{2}{\sqrt{\alpha}} \sum_{m=0}^{\infty} \frac{(-1)^{m} \Gamma(2m+2) \zeta^2(2m+2)}{(2\pi)^{2m+2}} \left( \frac{1}{2}\log(\alpha) +\gamma+ \psi(2m+2) + \frac{\zeta'(2m+2)}{\zeta(2m+2)} \right)\alpha^{2m+2}.
	 	\end{align}
	 	Similarly, from \eqref{ramanujan asymptotic}, as $\alpha\to0$, 
	 	\begin{align}\label{ramanujan asymptotic1}
	 		&\dfrac{1}{\pi^{3/2}}\int_0^{\i}\left|\Xi\left(\dfrac{1}{2}t\right)\Gamma\left(\dfrac{-1+it}{4}\right)\right|^2
	 		\dfrac{\cos\left(\tf{1}{2}t\log\alpha\right)}{1+t^2}\, dt\nonumber\\
	 		&\sim-\frac{\sqrt{\alpha}}{2}\left(\gamma-\log\left(\frac{2\pi}{\alpha}\right)\right)-\frac{2}{\sqrt{\alpha}}\sum_{m=1}^{\infty}\frac{(-1)^m}{(2\pi)^{2m}}\Gamma(2m)\zeta^2(2m)\alpha^{2m}.
	 	\end{align}
	 Equation \eqref{alpha zero} now follows from \eqref{neww}, \eqref{w1.26}, \eqref{ramanujan asymptotic1} and \eqref{zero alpha}  upon simplification.
	 	\end{proof}
	 	
	\section{Proofs of Carlitz-type transformations}

\begin{proof}[Theorem \textup{\ref{finitetrans2}}][]
	Differentiate \eqref{carlitz relation}  $k$ times with respect to $z$ using the Leibnitz rule to get
	\begin{align*}
		\sum_{\ell=0}^{k} \binom{k}{\ell} n^z \log^{k-\ell}(n) \sum_{j=1}^{m} \zeta^{(\ell)} \left(z, nx + \frac{n(j-1)}{m}\right) = \sum_{\ell=0}^{k} \binom{k}{\ell} m^z \log^{k-\ell}(m) \sum_{j=1}^{n} \zeta^{(\ell)} \left(z, mx + \frac{m(j-1)}{n}\right).
	\end{align*}
	Add and subtract the principal part of the Laurent series expansion of the $\ell$-th derivative of the Hurwitz zeta function from the summands of the inner sums on both sides  and rearrange the terms to get 
	\begin{align}
		& \sum_{\ell=0}^{k} \binom{k}{\ell} n^z \log^{k-\ell}(n) \sum_{j=1}^{m} \frac{(-1)^\ell \ell!}{(z-1)^{\ell+1}} - \sum_{\ell=0}^{k} \binom{k}{\ell} m^z \log^{k-\ell}(m)  \sum_{j=1}^{n}\frac{(-1)^\ell \ell!}{(z-1)^{\ell+1}} \notag  \\ 
		=& \sum_{\ell=0}^{k} \binom{k}{\ell} m^z \log^{k-\ell}(m)   \sum_{j=1}^{n} \left( \zeta^{(\ell)} \left(z, mx + \frac{m(j-1)}{n}\right) - \frac{(-1)^\ell  \ell!}{(z-1)^{\ell+1}} \right) \notag \\
		& -\sum_{\ell=0}^{k} \binom{k}{\ell} n^z \log^{k-\ell}(n)   \sum_{j=1}^{m} \left( \zeta^{(\ell)} \left(z, nx + \frac{n(j-1)}{m}\right) -  \frac{(-1)^\ell \ell!}{(z-1)^{\ell+1}} \right). \label{3.5.1}
	\end{align}
	The left-hand side of \eqref{3.5.1} can be simplified to
	\begin{align}
		&\sum_{\ell=0}^{k} \binom{k}{\ell} n^z \log^{k-\ell}(n) \sum_{j=1}^{m} \frac{(-1)^\ell \ell!}{(z-1)^{\ell+1}} - \sum_{\ell=0}^{k} \binom{k}{\ell} m^z \log^{k-\ell}(m)  \sum_{j=1}^{n}\frac{(-1)^\ell \ell!}{(z-1)^{\ell+1}} \notag  \\ 
		&= \frac{1}{(z-1)^{k+1}} \sum_{\ell=0}^{k}  (-1)^\ell \ell! \binom{k}{\ell} (z-1)^{k-\ell} \left\{ mn^z \log^{k-\ell}(n) - nm^z \log^{k-\ell}(m)  \right\}\notag\\
		&= \frac{R(z)}{(z-1)^{k+1}},\label{3.5.2}
	\end{align} 
	where $R(z)$ is defined by 
	\begin{align*}
		R(z):= \sum_{\ell=0}^{k}  (-1)^{k-\ell} (k-\ell)! \binom{k}{\ell} (z-1)^{\ell} \left( mn^z \log^{\ell}(n) - nm^z \log^{\ell}(m)  \right).
	\end{align*}
	We now show that for $0 \leq i \leq k, \lim_{z \to 1} \frac{d^i}{dz^i}	R(z) = 0.$
To that end, differentiate $R(z)$ $i$-times with respect to $z$ to get
	\begin{align*}
		 \frac{d^i}{dz^i}	R(z)
		= \sum_{\ell=0}^{k} (-1)^{k-\ell} (k-\ell)! \binom{k}{\ell} \sum_{r=0}^{i} \binom{i}{r} \frac{d^r}{dz^r} \left[ (z-1)^{\ell}  \right] \frac{d^{i-r}}{dz^{i-r}} \left[  mn^z \log^{\ell}(n) - nm^z \log^{\ell}(m)  \right].
	\end{align*}
	We wish to let $z \to 1$ on both sides.
	Note that for $r \neq \ell $, $ \lim_{z \to 1} \frac{d^r}{dz^r} (z-1)^{\ell} = 0$. So, only the term $r=\ell$ survives in the inside sum for each $0 \leq \ell \leq k$, which is possible only if $i \geq \ell$. Also for $\ell>i$, $\binom{i}{\ell}=0$, so the sum reduces to
	\begin{align*}
		\lim_{z \to 1} \frac{d^i}{dz^i}	R(z)=& \sum_{\ell=0}^{i} (-1)^{k-\ell} (k-\ell)! \binom{k}{\ell}  \binom{i}{\ell} \lim_{z \to 1} \frac{d^\ell}{dz^\ell} \left[ (z-1)^{\ell}  \right] \lim_{z \to 1} \frac{d^{i-\ell}}{dz^{i-\ell}} \left[  mn^z \log^{\ell}(n) - nm^z \log^{\ell}(m)  \right] \\
		=& \sum_{\ell=0}^{i} (-1)^{k-\ell} (k-\ell)! \binom{k}{\ell}  \binom{i}{\ell} \ell! \left[  mn \log^{i}(n) - nm \log^{i}(m)  \right] \\
		 =&0.
	\end{align*}
	Thus, we can apply L'Hospital's rule $k+1$ times to \eqref{3.5.2} to have
	\begin{align}
		\lim_{z \to 1}  \frac{R(z)}{(z-1)^{k+1}} 
		=& \frac{1}{(k+1)!} \sum_{\ell=0}^{k} (-1)^{k-\ell} (k-\ell)! \binom{k}{\ell} \binom{k+1}{\ell} \ell! \left[  mn \log^{k+1}(n) - nm \log^{k+1}(m)  \right]\notag  \\
		=& \frac{(-1)^k}{k+1} \left[  mn \log^{k+1}(n) - nm \log^{k+1}(m)  \right] \left\{ \sum_{\ell=0}^{k+1} (-1)^{\ell}  \binom{k+1}{\ell} - (-1)^{k+1} \right\} \notag \\
		=& \frac{1}{k+1} \left[  mn \log^{k+1}(n) - nm \log^{k+1}(m)  \right]. \label{3.5.3}
	\end{align}
	Take $\lim_{z \to 1}$ of both sides of \eqref{3.5.1}, and use \eqref{hurwitz_psi_j} and \eqref{3.5.3} to get
	\begin{align*}
		\frac{1}{k+1} \left[  mn \log^{k+1}(n) - nm \log^{k+1}(m)  \right]
		=&\sum_{\ell=0}^{k} \binom{k}{\ell} m \log^{k-\ell}(m)  \sum_{j=1}^{n} (-1)^{\ell+1} \psi_\ell\left(mx + \frac{m(j-1)}{n}\right) \\
		-&\sum_{\ell=0}^{k} \binom{k}{\ell} n \log^{k-\ell}(n)  \sum_{j=1}^{m} (-1)^{\ell+1} \psi_\ell\left(nx + \frac{n(j-1)}{m}\right).
	\end{align*}
	Re-arrange the terms of the above identity to arrive at \eqref{finitetrans2_eqn}, which completes the proof.
%
\end{proof}

\begin{proof}[Corollary \textup{\ref{carlitzcorollary}}][]
	Put $k=0$ in Theorem \ref{finitetrans2}.
\end{proof}

	\section{Proof of the inversion formula}\label{inv}
The inversion formula, which forms the basis of much of the developments in the sequel, is now proved.
\begin{proof}[Theorem \textup{\ref{Inv}}][]
	The proof invokes the principle of strong mathematical induction.
	
	Clearly $h(0)=1$, because the empty partition is the only partition of $0$.  Hence the Theorem holds true for $k=0$.
	Let  us assume that the statement holds true for all  $m, 0\leq m\leq k$, that is,
	\begin{align*}
		f(m)= \sum_{r=0}^{m}  \frac{(-1)^{r}}{s^{r+1}(1)} h(r) g(m-r),
	\end{align*}
	where $h(r)$ is defined in \eqref{h(r)}. We prove that
	\begin{align*}
		f(k+1)= \sum_{r=0}^{k+1}  \frac{(-1)^{r}}{s^{r+1}(1)} h(r) g(k+1-r).
	\end{align*}
	From the definition of $g(n)$,
	\begin{align*}
		g(k+1) &= \sum_{r=0}^{k+1} s(k-r+2) f(r),
		\end{align*}
which, with the help of the induction hypothesis, implies
\begin{align}
		-s(1)f(k+1) &= -g(k+1) + \sum_{r=0}^{k} s(k-r+2) \sum_{i=0}^{r}  \frac{(-1)^{i}}{s^{i+1}(1)} h(i) g(r-i)\notag\\
		 &= -g(k+1) + \sum_{i=0}^{k} \sum_{r=i}^{k} (-1)^{i} s(k-r+2) \frac{g(r-i)}{s^{i+1}(1)}h(i)\notag\\
	 &= -g(k+1) + \sum_{i=0}^{k} \sum_{y=0}^{k-i} (-1)^{i} s(k-y-i+2) \frac{g(y)}{s^{i+1}(1)}h(i)\notag\\
	&= -g(k+1) + \sum_{y=0}^{k} g(y) \sum_{i=0}^{k-y} (-1)^{i} \frac{s(k-y-i+2)}{s^{i+1}(1)}h(i),\label{Inv2}
	\end{align}
	where in the penultimate step, we let $r=y+i$. For $\ell\geq0$, we now show
	\begin{align} \label{Inv3}
		\sum_{i=0}^{\ell} (-1)^{i} \frac{s(\ell-i+2)}{s^{i+1}(1)}h(i) = \frac{(-1)^{\ell}}{s^{\ell+1}(1)}h(\ell+1).
	\end{align}
	Let $\mathscr{P}(2i)$ denote the set of all integer partitions of $2i$ in exactly $i$ parts, that is, the number of non-negative integer solutions of $1 b_1+...+ (r+1)b_{i+1}= 2i, b_1+...+b_{i+1}=i$. If $\pi$ denotes any such partition, then
	\begin{align}\label{justification}
		\sum_{i=0}^{\ell} (-1)^{i} \frac{s(\ell-i+2)}{s^{i+1}(1)}h(i)
		&= \sum_{i=0}^{\ell} (-1)^{i} \frac{s(\ell-i+2)}{s^{i+1}(1)} \sum_{\pi\in\mathscr{P}(2i)} (-1)^{b_1} \prod_{j=1}^{i+1} s^{b_j}(j) \frac{(i-{b_1})! }{b_2!...b_{i+1}!} \notag\\ 
		&= \frac{(-1)^\ell}{s^{\ell+1}(1)} \sum_{i=0}^{\ell} s^{\ell-i}(1) s(\ell-i+2) \sum_{\pi\in\mathscr{P}(2i)} (-1)^{b_1+\ell-i} \prod_{j=1}^{i+1} s^{b_j}(j) \frac{(i-{b_1})! }{b_2!...b_{i+1}!}\notag\\
		&= \frac{(-1)^\ell}{s^{\ell+1}(1)} \sum_{\pi'\in\mathscr{P}(2\ell+2)} (-1)^{c_1} \prod_{j=1}^{\ell+2} s^{c_j}(j) \frac{(\ell+1-{c_1})! }{c_2!...c_{l+2}!}\notag\\
		&=\frac{(-1)^\ell}{s^{\ell+1}(1)} h(\ell+1),
	\end{align}
	where the sum in the penultimate expression runs over all non-negative integer solutions of $1 c_1+...+ (\ell+2)c_{\ell+2}= 2\ell+2$ with $c_1+...+c_{\ell+2}=\ell+1$.
	
	The second-last step above can be justified as follows. For a non-negative integer $n$ and $\lambda\in\mathscr{P}(2n)$, define $\textup{freq}_{2n}(\lambda):=	\frac{(n-{b_1})! }{b_2!...b_{n+1}!}$. 
	Consider $\mathscr{P}'(2n)$ to be the multiset corresponding to the partitions $\pi\in\mathscr{P}(2n)$, where each $\pi$ occurs $\textup{freq}_{2n}(\pi)$ many times so that
\begin{align*}
&\sum_{i=0}^{\ell} s^{\ell-i}(1) s(\ell-i+2) \sum_{\pi\in\mathscr{P}(2i)} (-1)^{b_1+\ell-i} \prod_{j=1}^{i+1} s^{b_j}(j) \frac{(i-{b_1})! }{b_2!...b_{i+1}!} \notag\\
&=\sum_{i=0}^{\ell} \sum_{\pi\in\mathscr{P'}(2i)} (-1)^{b_1+\ell-i} s^{\ell-i}(1) s(\ell-i+2)\prod_{j=1}^{i+1} s^{b_j}(j).
\end{align*}
For $0\leq i\leq\ell$, let $\pi \in \mathscr{P}(2i)$. We add $\ell-i$ many $1$'s and the part $(\ell-i+2)$ to $\pi$ to get $\pi'$ which is now a partition of $2\ell+2$ in exactly $\ell+1$ parts. Moreover, for a fixed $\pi'\in\mathscr{P}(2\ell+2)$, that is, $\pi': 1 c_1+...+ (\ell+2)c_{\ell+2}= 2\ell+2$, where $c_1+...+c_{\ell+2}=\ell+1$, if $c_{\ell-i+2}>0$, then $\pi$ given by
$\pi: 1 (c_1-(\ell-i))+...+(\ell-i+2)(c_{\ell-i+2}-1)+\cdots+ (\ell+2)c_{\ell+2}= 2i$ belongs to $\mathscr{P}(2i)$. 

For the same fixed $\pi'$, we now need to evaluate exact count of $\pi$'s in $\mathscr{P}(2i)$, where $i$ varies from $0$ to $\ell$,  which can generate $\pi'$ on addition of $(\ell-i)$ many $1$'s and the part $(\ell-i+2)$. We show that this count is equal to $\textup{freq}_{2\ell+2}(\pi')$.
	
To that end, we can see that for each $i$ such that $0\leq i\leq\ell$, $\pi\in\mathscr{P}(2i)$ contributes 
\begin{equation*}
\textup{freq}_{2i}(\pi)=\frac{(i-(c_1-(\ell-i)))!}{c_2!\cdots(c_{\ell-i+2}-1)!\cdots c_{\ell+2}!}=\frac{(\ell-c_1)!}{c_2!\cdots c_{\ell+2}!}c_{\ell-i+2}
\end{equation*} 
towards the count irrespective of whether $c_{\ell-i+2}$ is non-zero or not. The second equality above stems from the fact that $c_j=0$ for any $j$ between $i+2$ and $\ell+2$. Since $c_1+...+c_{\ell+2}=\ell+1$, the total count of $\pi$'s as $i$ varies from $0$ to $\ell$ is
\begin{align*}
\sum_{i=0}^{\ell}\frac{(\ell-c_1)!}{c_2!\cdots c_{\ell+2}!}c_{\ell-i+2}=\frac{(\ell+1-c_1)!}{c_2!\cdots c_{\ell+2}!}=\textup{freq}_{2\ell+2}(\pi'),
\end{align*}
by definition. In other words, the cardinality of $\cup_{i=1}^{\ell}\mathscr{P'}(2i)$ is equal to that of $\mathscr{P'}(2\ell+2)$. Therefore, 
\begin{align*}
\sum_{i=0}^{\ell} \sum_{\pi\in\mathscr{P'}(2i)} (-1)^{b_1+\ell-i} s^{\ell-i}(1) s(\ell-i+2)\prod_{j=1}^{i+1} s^{b_j}(j)
&=\sum_{\pi'\in\mathscr{P'}(2\ell+2)} (-1)^{c_1} \prod_{j=1}^{\ell+2} s^{c_j}(j)\notag\\ &=\sum_{\pi'\in\mathscr{P}(2\ell+2)} (-1)^{c_1} \prod_{j=1}^{\ell+2} s^{c_j}(j) \frac{(\ell+1-{c_1})! }{c_2!...c_{l+2}!},
\end{align*}
which justifies the third equality in \eqref{justification}.

	Now substitute $\ell=k-y$ in \eqref{Inv3}, substitute the resultant in \eqref{Inv2} and simplify to obtain
	\begin{align*}
		f(k+1) &= \sum_{y=0}^{k+1} \frac{(-1)^{k+1-y}}{s^{k+2-y}(1)}  h(k+1-y) g(y)\\
	 &= \sum_{y=0}^{k+1} \frac{(-1)^{y}}{s^{y+1}(1)} h(y) g(k+1-y).
	\end{align*}
	By invoking the principle of mathematical induction, we see that the proof is now complete.
\end{proof}

	\section{Combinatorics meets Special Functions: Proof of Theorem \ref{finitetrans1}}
	
Transformations of Guinand-type, that is, \eqref{guigen}, containing second or higher-order derivatives of $\psi_k(x), k\geq0,$ crucially involve the combinatorial object $h(r)$ dealt with in Section \ref{inv}. This can be seen in the proof of Theorem \ref{finitetrans1} which we now give. Let
	\begin{align*}
		F_{m,n}(z) := \left( \frac{m}{n}	\right) ^{\frac{-z}{2}}	\sum_{j=1}^{m} \zeta \left(z, nx + \frac{n(j-1)}{m}\right).
	\end{align*}
	Then from \eqref{carlitz relation},
	\begin{align}
	F_{m, n}(z)=F_{n,m}(z). \label{mnnm}
	\end{align}
Differentiate  $F_{m, n}(z)$ $k$ times with respect to $z$ to obtain
	\begin{align}
		\frac{d^k}{dz^k} (F_{m,n}(z)) = \sum_{\ell=0}^{k} \binom{k}{\ell} \left( \frac{m}{n}	\right) ^{\frac{-z}{2}} \left(\frac{-1}{2}\right)^{k-\ell} \log^{k-\ell}\left(\frac{m}{n}\right) \sum_{j=1}^{m} \zeta^{(\ell)}\left(z,nx + \frac{n(j-1)}{m}\right). \label{Inv5}
	\end{align}
	From \cite[Corollary 2]{coffey} and \eqref{psigamma}, we have for $z\in\mathbb{N}, z>1$,
	\begin{align}\label{coffey special}
		 -\frac{\psi_\ell^{(z-1)} (y)}{\ell!} = \sum_{r=0}^{\ell} s(z, \ell-r+1) \frac{(-1)^{r}}{r!} \zeta^{(r)}(z,y) .
	\end{align}
	We use Theorem \ref{Inv} with $g(\ell) = -\frac{\psi_\ell^{(z-1)} (y)}{\ell!}$, $ f(r) = \frac{(-1)^{r}}{r!} \zeta^{(r)}(z,y)$, and $s(i)=s(z,i)$, the Stirling number of the first kind, to get
	\begin{align}
		\zeta^{(\ell)}(z,y) &= \sum_{r=0}^{\ell} \frac{(-1)^{\ell+r+1}}{s^{r+1}(z,1)} \frac{\ell!}{(\ell-r)!} h(r) \psi_{\ell-r}^{(z-1)} (y).  \label{Inv6} 
	\end{align}
	where $h(r)$ is defined in \eqref{h(r)}. Now substitute \eqref{Inv6} in \eqref{Inv5} to have
	\begin{align} 
		\frac{d^k}{dz^k} (F_{m,n}(z)) = \sum_{\ell=0}^{k} \binom{k}{\ell} \left( \frac{m}{n}	\right) ^{\frac{-z}{2}} \log^{k-\ell}\left(\sqrt{\frac{n}{m}}\right) \sum_{j=1}^{m} \sum_{r=0}^{\ell}  \frac{(-1)^{\ell+r+1}}{s^{r+1}(z,1)} \frac{\ell!}{(\ell-r)!} h(r) \psi_{\ell-r}^{(z-1)} \left(nx + \frac{n(j-1)}{m}\right). \label{Inv7}
	\end{align}
	From \eqref{mnnm} and \eqref{Inv7}, we arrive at \eqref{meeting equation} upon simplification. This completes the proof of Theorem \ref{finitetrans1}.
	
	\qed

\begin{proof}[Corollary \textup{\ref{finitetrans1 cor1}}][]
	Take $z=2, k=1, m=1$ and $n=2$ in Theorem \ref{finitetrans1} and use the facts that $s(2, 2)=1$ and $h(0)=h(1)=1$.
\end{proof}	

\begin{proof}[Corollary \textup{\ref{finitetrans1 cor2}}][]
		Take $z=2, k=2, m=2$ and $n=1$ in Theorem \ref{finitetrans1}, use the facts $s(2,1)=-1$, $s(2,2)=1$, $s(2,3)=0$, and $h(0)=h(1)=h(2)=1$. 
\end{proof}

	\section{Proofs of Guinand-type transformations}
	We begin this section with a couple of lemmas, which are also of independent interest. The second one uses the first, and justifies why the infinite series of 	$\psi_{\ell}^{(z-1)}$ in our theorems are convergent. 
	\begin{lemma} \label{lemma minus two}
		Let $r\in\mathbb{N}\cup\{0\}$, $x>0$ and \textup{Re}$(z)>-1, z\neq 1$. As $x\to\infty$,
		\begin{align}\label{minus two}
			\zeta^{(r)}(z, x)&=\sum_{t=0}^{r}\binom{r}{t}\frac{(-1)^rt!}{(z-1)^{t+1}}\frac{\log^{r-t}(x)}{x^{z-1}}+(-1)^{r}\frac{\log^{r}(x)}{2x^{z}}+O_{r, z}\left(\frac{\log^{r}(x)}{x^{\textup{Re}(z)+1}}\right).
				\end{align}
	\end{lemma}
	\begin{proof}
		From \cite[p.~608, formula \textbf{25.11.20}]{nist}, for $x>0$ and \textup{Re}$(z)>-1, z\neq 1$,
		\begin{align}\label{minus one}
			\zeta^{(r)}(z, x)&=\sum_{t=0}^{r}\binom{r}{t}\frac{(-1)^rt!}{(z-1)^{t+1}}\frac{\log^{r-t}(x)}{x^{z-1}}+(-1)^{r}\frac{\log^{r}(x)}{2x^{z}}+(-1)^{r+1}z(z+1)\int_{0}^{\infty}\frac{\tilde{B_2}(u)\log^r(u+x)}{(u+x)^{z+2}}\, du\notag\\
			&\quad+r(-1)^r(2z+1)\int_{0}^{\infty}\frac{\tilde{B_2}(u)\log^{r-1}(u+x)}{(u+x)^{z+2}}\, du+(-1)^{r+1}r(r-1)\int_{0}^{\infty}\frac{\tilde{B_2}(u)\log^{r-2}(u+x)}{(u+x)^{z+2}}\, du,
		\end{align}
		where $\tilde{B_2}(u)$ denotes the second periodized Bernoulli polynomial \cite[p.~5-6]{temme}, that is,
		\begin{align*}
			\tilde{B_2}(u)=B_2(u)&=u^2-u+\frac{1}{6}, 0\leq u\leq1,\notag\\
			\tilde{B_2}(u+1)&=\tilde{B_2}(u), u\in\mathbb{R}.
		\end{align*}
		
Thus, it suffices to show that for any $n\in\mathbb{N}\cup\{0\}$, as $x\to\infty$, we have
\begin{align}\label{zero}
\int_{0}^{\infty}\frac{\tilde{B_2}(u)\log^n(u+x)}{(u+x)^{z+2}}\, du=O_z\left(\frac{\log^{n}(x)}{x^{\textup{Re}(z)+1}}\right).
\end{align}
To that end, observe that since $\tilde{B_2}(u)\leq13/6$ on $(0,\infty)$, 
\begin{align}\label{one}
\left|\int_{0}^{\infty}\frac{\tilde{B_2}(u)\log^n(u+x)}{(u+x)^{z+2}}\, du\right|\leq\frac{13}{6}\int_{0}^{\infty}\frac{\log^n(u+x)}{(u+x)^{\textup{Re}(z)+2}}\, du.
\end{align}
Now let $(\textup{Re}(z)+1)\log(u+x)=y$ so that $du/(u+x)=dy/(\textup{Re}(z)+1)$. Thus,
\begin{align}\label{two}
\int_{0}^{\infty}\frac{\log^n(u+x)}{(u+x)^{\textup{Re}(z)+2}}\, du
=\frac{1}{(\textup{Re}(z)+1)^{n+1}}\Gamma(n+1, (\textup{Re}(z)+1)\log(x)),
\end{align}
where $\Gamma(a, w):=\int_{w}^{\infty}e^{-y}y^{a-1}\, dy$ is the incomplete gamma function. From \cite[p.~179, Formula \textbf{8.11.2}]{nist}, for a fixed $a$, we have $\Gamma(a, w)\sim w^{a-1}e^{-w}$ as $w\to\infty$. Employing this in \eqref{two}, using the resultant in \eqref{one}, and then simplifying, we arrive at \eqref{zero}. 

Finally, using \eqref{zero} three times, with $n=r, r-1$ and $r-2$ to handle the three integrals in \eqref{minus one}, we obtain \eqref{minus two}.		
	\end{proof}

	\begin{lemma}\label{8.2a}
		Let $x>0$, $z\in\mathbb{N}, z>1$ and let $\ell\in\mathbb{N}\cup\{0\}$. As  $x\to\infty$,
\begin{align}\label{8.7ell}
	\psi_{\ell}^{(z-1)}(x)=-\ell! \sum_{r=0}^{\ell} \frac{s(z, \ell-r+1)}{r!} \left\{\sum_{t=0}^{r}\binom{r}{t}\frac{t!}{(z-1)^{t+1}} \frac{\log^{r-t}(x)}{x^{z-1}}+\frac{\log^{r}(x)}{2x^{z}}\right\}+O_{\ell, z}\left(\frac{\log^{\ell}(x)}{x^{\textup{Re}(z)+1}}\right),
		\end{align}
		where $s(a, b)$ is the Stirling number of the first kind. In particular, as $x\to\infty$,
		\begin{align}\label{psil dash}
				\psi_{\ell}'(x)= \frac{\log^{\ell}(x)}{x} + O_{\ell}\left( \frac{\log^{\ell}(x)}{x^2} \right).
		\end{align}
	\end{lemma}
	\begin{proof}
Equation \eqref{8.7ell} results from using Lemma \ref{lemma minus two} in \eqref{coffey special} and simplifying. As for \eqref{psil dash}, we let $z=2$ in \eqref{8.7ell} and observe that only the terms corresponding to $r=\ell$ and $r=\ell-1$ survive, since $s(2, k)=0$ for $k>2$. Using the facts that $s(2, 1)=-1$ and $s(2, 2)=1$, this simplifies to 
 \begin{align*}
 		\psi_{\ell}'(x)=\frac{\ell!}{x}\left\{\sum_{t=0}^{\ell}\frac{\log^{\ell-t}(x)}{(\ell-t)!}-\sum_{t=0}^{\ell-1}\frac{\log^{\ell-1-t}(x)}{(\ell-1-t)!}\right\}+ O_{\ell}\left( \frac{\log^{\ell}(x)}{x^2} \right)=\frac{\log^{\ell}(x)}{x} + O_{\ell}\left( \frac{\log^{\ell}(x)}{x^2} \right).
 \end{align*}
	\end{proof}
	\begin{remark}
		The result in the above lemma can also be obtained by differentiating the Plana-type integral for $\psi_k(x)$ \cite[Theorem 2]{ishibashi} and then using the generalization of Watson's lemma given in Theorem \ref{watsongen}.
	\end{remark}
	\subsection{The case $z\in\mathbb{N}, z>2$}	
We are now ready to prove Theorem \ref{guinand gen psi_k z>2}.

	Replace $m$ and $n$ by $mN$ and $nN$ respectively  in Theorem \ref{finitetrans1}, let $s(i)=s(z,i)$ in $h(r)$, $x=\frac{1}{N}\left( \frac{1}{n}+ \frac{1}{m} \right)$, $\frac{m}{n}=\alpha$, $\frac{n}{m}=\beta$. Then let  $N \to \infty$ and observe from Lemma \ref{8.2a} that the resulting infinite series converge. This leads us to \eqref{z>2} for $\alpha, \beta \in \mathbb{Q}^{+}$. 
	
	Note that from \eqref{coffey special}, $\psi_{\ell-r}^{(z-1)} \left(1 + \frac{j}{\alpha}\right)$ is continuous for $\alpha>0$. Also, $\sum\limits_{j=1}^{\infty}  \psi_{\ell-r}^{(z-1)} \left(1 + \frac{j}{\alpha}\right)$ is the uniform limit of continuous functions as can be seen from Lemma \ref{8.2a}. Hence by the uniform limit theorem, $\sum\limits_{j=1}^{\infty}  \psi_{\ell-r}^{(z-1)} \left(1 + \frac{j}{\alpha}\right)$ (and similarly, $\sum\limits_{j=1}^{\infty}  \psi_{\ell-r}^{(z-1)} \left(1 + j\alpha\right)$) is continuous for $\alpha>0$. This implies that \eqref{z>2} holds for $\alpha, \beta>0$.
	
\qed
	
	\subsection{The case $z=2$} We begin with obtaining an asymptotic estimate for a summatory function which is used in the sequel.
\begin{lemma} \label{estimatelog}
	For $x,y>0$ and $j \in \mathbb{N} \cup \{ 0 \}$, as $x \to \infty$, we have 
	\begin{align*}
		\sum_{n \leq x} \frac{\log^{j}(ny)}{n} &=   \sum_{t=0}^{j} \binom{j}{t} \log^{j-t}(y) \left( \frac{\log^{t+1}(x)}{t+1}+\gamma_t  \right) + O_{y} \left( \frac{\log^j(x)}{x} \right) .
	\end{align*}
\end{lemma}
\begin{proof}
	For $x>0$ and $m \in \mathbb{N} \cup \{ 0 \}$,
	\begin{align}
		\sum_{n \leq x} \frac{\log^{m}(n)}{n} 
		&= \sum_{n=1}^{\infty} \left( \frac{\log^{m}(n)}{n} - \frac{\log^{m}(n+ \lfloor x \rfloor)}{n+ \lfloor x \rfloor} \right)  \notag\\
		&= \psi_m(\lfloor x \rfloor + 1) +\gamma_m  \notag \\
		&= \frac{\log^{m+1}(x)}{m+1}+\gamma_m + O \left( \frac{\log^m(x)}{x} \right),\label{AsyLog}
	\end{align}
	where in the last step, we have used \eqref{asymptotic coffey}.  We now consider a slightly different finite sum, for $y>0$:
	\begin{align}
		\sum_{n \leq x} \frac{\log^{j}(ny)}{n} &= \sum_{n\leq x} \frac{\left( \log(n)+\log(y) \right)^j}{n} 
		=  \sum_{t=0}^{j} \binom{j}{t} \log^{j-t}(y) \sum_{n\leq x} \frac{\log^t(n)}{n}. \label{FinAsy} 	
	\end{align}
	Use \eqref{AsyLog} in \eqref{FinAsy} to get
	\begin{align*}
		\sum_{n \leq x} \frac{\log^{j}(ny)}{n} &=   \sum_{t=0}^{j} \binom{j}{t} \log^{j-t}(y) \left( \frac{\log^{t+1}(x)}{t+1}+\gamma_t + O \left( \frac{\log^t(x)}{x} \right) \right),
	\end{align*}
	which completes the proof.
\end{proof}

\begin{proof}[Theorem \textup{\ref{curious}}][]
	Replace $m$ and $n$ by $mN$ and $nN$ respectively  in Theorem \ref{finitetrans1}, then let $z=2$, $s(i)=s(2,i)$ in $h(r)$ to have $h(r)=1$ for all $r \in \mathbb{N} \cup \{ 0\}$, and let $x=\frac{1}{N}\left( \frac{1}{n}+ \frac{1}{m} \right)$ to obtain 
	\begin{align}\label{1curious}
		& \frac{n}{m}\sum_{\ell=0}^{k} \sum_{r=0}^{\ell} \frac{2^\ell}{(k-\ell)!(\ell-r)!} \log^{k-\ell} \left( \frac{m}{n} \right) \sum_{j=1}^{mN} \psi_{\ell-r}' \left(1 + \frac{nj}{m}\right) \notag\\
		=& \frac{m}{n}\sum_{\ell=0}^{k} \sum_{r=0}^{\ell} \frac{2^\ell}{(k-\ell)!(\ell-r)!}\log^{k-\ell} \left( \frac{n}{m} \right) \sum_{j=1}^{nN} \psi_{\ell-r}' \left(1 + \frac{mj}{n}\right).
	\end{align} 
	We cannot let $N\to\infty$ directly in the above result since that would render the two sums over $j$ divergent. This is now explained.
	 	
	 		Differentiate \eqref{dilcher fe} with respect to $x$, invoke \eqref{psil dash} and then simplify to obtain 
	 		\begin{align*}
	 				\psi_{k}'(1+x) = \frac{\log^{k}(x)}{x} +  O \left( \frac{\log^{k}(x)}{x^{2}} \right).
	 			\end{align*}
	 			This implies that $\sum_{j=1}^{\infty} \psi_{\ell-r}' \left(1 + \frac{nj}{m}\right)$ diverges. Hence we add and subtract $\frac{\log^{\ell-r}(\frac{nj}{m})}{\frac{nj}{m}}$ from the summands of the finite sums over $j$ on both sides of \eqref{1curious} (so that later we can let $N\to\infty$), and rearrange to obtain
	\begin{align}
		&\frac{n}{m} \sum_{\ell=0}^{k} \sum_{r=0}^{\ell} \frac{2^\ell}{(k-\ell)!(\ell-r)!} \log^{k-\ell} \left( \frac{m}{n} \right) \sum_{j=1}^{mN} \left[ \psi_{\ell-r}' \left(1 + \frac{nj}{m}\right) - \frac{\log^{\ell-r}(\frac{nj}{m})}{\frac{nj}{m}} \right] \notag \\
		&- \frac{m}{n} \sum_{\ell=0}^{k} \sum_{r=0}^{\ell} \frac{2^\ell}{(k-\ell)!(\ell-r)!} \log^{k-\ell} \left( \frac{n}{m} \right) \sum_{j=1}^{nN} \left[ \psi_{\ell-r}' \left(1 + \frac{mj}{n}\right) - \frac{\log^{\ell-r}(\frac{mj}{n})}{\frac{mj}{n}} \right]\notag\\
	&=	  \sum_{\ell=0}^{k} \sum_{r=0}^{\ell} \frac{2^\ell T(m, n, k)}{(k-\ell)!(\ell-r)!} , 	  \label{maindiff1}
		\end{align}
		where $T(m,n,k)$ is defined by 
	\begin{align*}
		T(m, n, k):=\left( \frac{m}{n} \right) \log^{k-\ell} \left( \frac{n}{m} \right) \sum_{j=1}^{nN} \frac{\log^{\ell-r}(\frac{mj}{n})}{\frac{mj}{n}} - \left( \frac{n}{m} \right) \log^{k-\ell} \left( \frac{m}{n} \right) \sum_{j=1}^{mN} \frac{\log^{\ell-r}(\frac{nj}{m})}{\frac{nj}{m}}.
	\end{align*}
	Use Lemma \ref{estimatelog} with $x=nN$, $y=\frac{m}{n}$ and $j=\ell-r$ in the first sum and $x=mN$, $y=\frac{n}{m}$ and $j=\ell-r$ in the second sum, and simplify to obtain
	\begin{align}
		T(m, n, k)&=
		(-1)^{k-\ell} \sum_{t=0}^{\ell-r} \binom{\ell-r}{t} \left[ \log^{k-r-t} \left( \frac{m}{n} \right) \left\{ \frac{\log^{t+1}(nN)}{t+1}+\gamma_t  \right\} \right.\nonumber\\
	&	\left.\quad- \log^{k-r-t} \left( \frac{n}{m} \right) \left\{ \frac{\log^{t+1}(mN)}{t+1}+\gamma_t  \right\} \right] + O \left( \frac{\log^{\ell-r}(N)}{N} \right). \label{diffsimp1}
	\end{align}
Therefore, using \eqref{diffsimp1}, we have
	\begin{align}
		\sum_{\ell=0}^{k} \sum_{r=0}^{\ell} \frac{2^\ell T(m, n, k)}{(k-\ell)!(\ell-r)!}
		&= \sum_{\ell=0}^{k} \sum_{r=0}^{\ell} \frac{(-1)^{k-\ell}2^\ell }{(k-\ell)!(\ell-r)!}  \sum_{t=0}^{\ell-r} \binom{\ell-r}{t} \left[ \log^{k-r-t} \left( \frac{m}{n} \right) \left\{ \frac{\log^{t+1}(nN)}{t+1}+\gamma_t  \right\} \right. \notag \\ 
		& \quad-\left.  \log^{k-r-t} \left( \frac{n}{m} \right) \left\{ \frac{\log^{t+1}(mN)}{t+1}+\gamma_t  \right\} \right]   + O \left( \frac{\log^{k}(N)}{N} \right)\notag\\
		&=f_1(m,n,k)-f_1(n,m,k)+f_2(m,n,k)-f_2(n,m,k)+O \left( \frac{\log^{k}(N)}{N} \right),\label{T as difference of f1 and f2}
	\end{align}
	where
	\begin{align}
		&f_1(m,n,k):= \sum_{\ell=0}^{k} \sum_{r=0}^{\ell} \frac{2^\ell (-1)^{k-\ell}}{(k-\ell)!(\ell-r)!}  \sum_{t=0}^{\ell-r} \binom{\ell-r}{t}  \log^{k-r-t} \left( \frac{m}{n} \right) \left\{ \frac{\log^{t+1}(nN)}{t+1} \right\}, \label{defn of f1} \\
		&f_{2}(m,n,k):=\sum_{\ell =0}^{k} \sum_{r=0}^{\ell} \frac{2^\ell (-1)^{k-\ell}}{(k-\ell)!(\ell-r)!} \sum_{t=0}^{\ell-r} \binom{\ell-r}{t} \log^{k-r-t} \left( \frac{m}{n} \right) \gamma_t. \label{defn of f2} 
	\end{align}

	We first simplify $f_1$. Consider the inner sum over $t$, that is,
	\begin{align}\label{inner sum of f1}
		&\sum_{t=0}^{\ell-r} \binom{\ell-r}{t}  \log^{k-r-t} \left( \frac{m}{n} \right)  \frac{\log^{t+1}(nN)}{t+1} \notag\\
		=& \sum_{t=0}^{\ell-r} \binom{\ell-r}{t}  \log^{k-r-t} \left( \frac{m}{n} \right) \int_{1}^{nN} \frac{\log^t(x)}{x} \, dx \notag\\
		=& \log^{k-\ell} \left( \frac{m}{n} \right) \int_{1}^{nN} \frac{1}{x} \sum_{t=0}^{\ell-r} \binom{\ell-r}{t}  \log^{\ell-r-t} \left( \frac{m}{n} \right)  \log^{t}(x) \, dx \notag\\
		=& \log^{k-\ell} \left( \frac{m}{n} \right) \int_{1}^{nN}  \log^{\ell-r} \left( \frac{mx}{n} \right)\frac{dx}{x}\notag\\
		=& \frac{\log^{k-\ell}\left( \frac{m}{n} \right)}{\ell-r+1}  \left[ \log^{\ell-r+1} \left( mN \right)  - \log^{\ell-r+1} \left( \frac{m}{n} \right)  \right].
	\end{align}
	Substitute \eqref{inner sum of f1} in \eqref{defn of f1} to see that
		\begin{align}\label{f1 simplified}
		f_1(m,n,k)&= \sum_{\ell=0}^{k} \sum_{r=0}^{\ell} \frac{2^\ell (-1)^{k-\ell} \log^{k-\ell}\left( \frac{m}{n} \right)}{(k-\ell)!(\ell-r+1)!} \left[ \log^{\ell-r+1} \left( mN \right)  - \log^{\ell-r+1} \left( \frac{m}{n} \right)  \right]\nonumber\\
&=f_{11}(m,n,k)+	f_{12}(m,n,k),
\end{align}
where
	\begin{align}
	f_{11}(m,n,k)&:= \sum_{\ell=0}^{k} \sum_{r=0}^{\ell} \frac{2^\ell (-1)^{k-\ell} \log^{k-\ell}(\frac{m}{n})}{(k-\ell)!(\ell-r+1)!} \log^{\ell -r+1} (mN),\label{f11}\\
	f_{12}(m,n,k)&:=\sum_{\ell =0}^{k} \sum_{r=0}^{\ell} \frac{2^\ell (-1)^{k-\ell +1}\log^{k-r+1}\left(\frac{m}{n}\right)}{(k-\ell)!(\ell -r+1)!}.\label{f12}
\end{align}
We now simplify $f_{11}$ and $f_{12}$ separately. Interchange the order of summation in \eqref{f11}, then let $\ell-r=j$ to get
	\begin{align*}
		f_{11}(m,n,k)&= \sum_{j=0}^{k} \sum_{r=0}^{k-j} \frac{2^{r+j}(-1)^{k-r-j}\log^{k-r-j}(\frac{m}{n})}{(k-r-j)!(j+1)!}\log^{j+1}(mN) \\
		&= \sum_{j=0}^{k} \sum_{r=0}^{k-j} \frac{2^{r+j}\log^{k-r-j}(\frac{n}{m})}{(k-r-j)!(j+1)!} \sum_{s=0}^{j+1} \binom{j+1}{s} \log^s(N) \log^{j+1-s}(m).
	\end{align*}
	Hence, 
	\begin{align*}
		&f_{11}(m,n,k)- f_{11}(n,m,k) \notag \\ 
		&= \sum_{j=0}^{k} \sum_{r=0}^{k-j} \sum_{s=0}^{j+1} \frac{2^{r+j}\left( \log^{k-r-j}\left(\tfrac{n}{m}\right) \log^{j+1-s}(m) - \log^{k-r-j}\left( \tfrac{m}{n} \right) \log^{j+1-s}(n) \right) }{(k-r-j)!(j+1-s)!s!}  \log^s(N) 
	\end{align*}
 is a polynomial of $\log(N)$ of degree less than or equal to $k+1$. We now show that the coefficient of $\log^y(N)$ is zero, for any fixed $y$ such that $ 1\leq y \leq k+1$. To that end, note that the terms corresponding to $j \leq y-2$ do not contribute to the coefficient of $\log^y(N)$; thus the coefficient of $\log^y(N)$ is given by
	\begin{align*}
		&\sum_{j=y-1}^{k} \sum_{r=0}^{k-j} \frac{2^{r+j}}{(k-r-j)!(j+1-y)!y!} \left[ \log^{k-r-j}\left( \frac{n}{m} \right) \log^{j+1-y}(m) - \log^{k-r-j}\left( \frac{m}{n} \right) \log^{j+1-y}(n) \right]\notag\\
		&=\frac{1}{y!} \sum_{r=0}^{k-y-1} \sum_{j=y-1}^{k-r} \frac{2^{r+j}}{(k-r-j)!(j+1-y)!} \left[ \log^{k-r-j}\left( \frac{n}{m} \right) \log^{j+1-y}(m) - \log^{k-r-j}\left( \frac{m}{n} \right) \log^{j+1-y}(n) \right]\notag\\ 
		&=\frac{1}{y!} \sum_{r=0}^{k-y-1} \sum_{t=0}^{k-r-y+1} \frac{2^{r+t+y-1}}{(k-r-t-y+1)!t!} \left[ \log^{k-r-t-y+1}\left(\frac{n}{m}\right)\log^{t}(m) - \log^{k-r-t-y+1}\left(\frac{m}{n}\right)\log^{t}(n) \right] \notag \\
		&= \frac{1}{y!} \sum_{r=0}^{k-y-1} \frac{2^{r+y-1}}{(k-r-y+1)!} \left[ \left( \log \left( \frac{n}{m} \right) + 2 \log(m) \right)^{k-r-y+1} - \left( \log \left( \frac{m}{n} \right) + 2 \log(n) \right)^{k-r-y+1} \right] \notag \\
		&= 0,
	\end{align*}
	where, in the second step above, we employed the change of variable $t=j-y+1$. Hence the polynomial $f_{11}(m,n,k)- f_{11}(n,m,k)$ reduces to the constant
	\begin{align*}
		&\sum_{j=0}^{k} \sum_{r=0}^{k-j} \frac{2^{r+j}}{(k-r-j)!(j+1)!} \left[ \log^{k-r-j}\left( \frac{n}{m} \right) \log^{j+1}(m) - \log^{k-r-j}\left( \frac{m}{n} \right) \log^{j+1}(n) \right] ,
	\end{align*}
	which, upon the interchange of the order of summation, and the substitution $t=j+1$ results in
	\begin{align*}
 \sum_{r=0}^{k} \sum_{t=1}^{k-r+1}  \frac{2^{r+t-1}}{(k-r-t+1)!t!} \left[ \log^{k-r-t+1}\left( \frac{n}{m} \right) \log^{t}(m) - \log^{k-r-t+1}\left( \frac{m}{n} \right) \log^{t}(n) \right].
 \end{align*}
 Simplifying the above expression in a similar way as in the previous calculation, we see that 
 \begin{align}
 	f_{11}(m,n,k)- f_{11}(n,m,k)
		= - \sum_{r=0}^{k} \frac{2^{k-r-1}}{(r+1)!} \left[ \log^{r+1}\left( \frac{n}{m} \right) - \log^{r+1}\left( \frac{m}{n} \right) \right].  \label{f11diff}
	\end{align}
We now work on $f_{12}(m,n,k)$. Interchange the order of summation in \eqref{f12} and simplify to get
	\begin{align*}
	f_{12}(m,n,k)	
		= (-1)^{k+1} \sum_{r=0}^{k} \frac{\log^{k-r+1}\left(\frac{m}{n}\right)}{(k-r+1)!} \sum_{\ell=r}^{k} \binom{k-r+1}{k-\ell}(-2)^\ell .
	\end{align*}
	Let $j=\ell-r$ and then replace $r$ by $k-r$ to have
	\begin{align*}
	f_{12}(m,n,k)		&= (-1)^{k+1} \sum_{r=0}^{k} \frac{\log^{k-r+1}\left(\frac{m}{n}\right)}{(k-r+1)!} \sum_{j=0}^{k-r} \binom{k-r+1}{k-j-r}(-2)^{j+r}\notag\\
	&= (-1)^{k+1} \sum_{r=0}^{k} \frac{\log^{r+1}\left(\frac{m}{n}\right)}{(r+1)!} \sum_{j=0}^{r} \binom{r+1}{r-j}(-2)^{k+j-r} \notag \\
		&= (-1)^{k+1} \sum_{r=0}^{k} \frac{(-2)^{k-r}\log^{r+1}\left(\frac{m}{n}\right)}{(r+1)!} \sum_{j=0}^{r+1} \binom{r+1}{r-j}(-2)^{j}\notag\\
		&= (-1)^{k+1} \sum_{r=0}^{k} \frac{(-2)^{k-r}\log^{r+1}\left(\frac{m}{n}\right)}{(r+1)!} \left( \frac{1+(-1)^r}{2} \right).
	\end{align*}  
	Therefore,
	\begin{align}
		f_{12}(m,n,k)-f_{12}(n,m,k)
		=2 \sum_{r=0}^{k} \frac{2^{k-r-1}}{(r+1)!} \left[ \log^{r+1}\left(\frac{n}{m}\right) - \log^{r+1}\left(\frac{m}{n}\right) \right]. \label{f12diff}
	\end{align}
	We next proceed to simplify $f_2(m, n, k)$ defined in \eqref{defn of f2}. Interchange the order of summation over $r$ and $t$ and simplify to get
	\begin{align*}
	f_2(m, n, k)
		&= \sum_{\ell =0}^{k} \sum_{t=0}^{\ell} \gamma_t \frac{2^\ell (-1)^{k-\ell}}{(k-\ell)!t!} \log^{k-t} \left( \frac{m}{n} \right) \sum_{r=0}^{\ell-t} \frac{\log^{-r} \left( \frac{m}{n} \right)}{(\ell-r-t)!}\notag\\
		&= \sum_{t=0}^{k} \frac{\gamma_t}{t!} \sum_{\ell=t}^{k} \frac{2^\ell (-1)^{k-\ell}}{(k-\ell)!} \log^{k-t} \left( \frac{m}{n} \right) \sum_{r=0}^{\ell-t} \frac{\log^{-r} \left( \frac{m}{n} \right)}{(\ell-r-t)!}\notag\\
		&= \sum_{t=0}^{k} \frac{\gamma_t}{t!} \sum_{j=0}^{k-t} \frac{2^{j+t} (-1)^{k-j-t}}{(k-j-t)!} \log^{k-t} \left( \frac{m}{n} \right) \sum_{r=0}^{j} \frac{\log^{-r} \left( \frac{m}{n} \right)}{(j-r)!}\notag\\
		&= \sum_{t=0}^{k} \frac{\gamma_t}{t!} \sum_{r=0}^{k-t} 2^t (-1)^{k-t} \log^{k-t-r} \left( \frac{m}{n} \right) \sum_{j=r}^{k-t} \frac{(-2)^j}{(k-j-t)!(j-r)!},
	\end{align*}
		where, in the third step above, we let $j=\ell-t$.	Let $\ell=j-r$ in the innermost sum above and simplify using the binomial theorem to arrive at
	\begin{align}\label{f2diff}
		f_2(m, n, k)
		= \sum_{t=0}^{k} \frac{\gamma_t}{t!} \sum_{r=0}^{k-t} \frac{2^{t+r} }{(k-t-r)!} \log^{k-t-r} \left( \frac{m}{n} \right).
	\end{align}
	Now substitute \eqref{f1 simplified}, \eqref{f11diff}, \eqref{f12diff} and \eqref{f2diff} in \eqref{T as difference of f1 and f2} and simplify so as to obtain
	\begin{align}
			&\sum_{\ell=0}^{k} \sum_{r=0}^{\ell} \frac{2^\ell T(m, n, k)}{(k-\ell)!(\ell-r)!}\notag\\
			&=\sum_{r=0}^{k} \frac{2^{k-r-1}}{(r+1)!} \left[ \log^{r+1}\left( \frac{n}{m} \right) - \log^{r+1}\left( \frac{m}{n} \right) \right] \notag\\
			&\quad+  \sum_{t=0}^{k} \frac{\gamma_t}{t!} \sum_{r=0}^{k-t} \frac{2^{t+r} }{(k-t-r)!} \left[ \log^{k-t-r} \left( \frac{m}{n} \right) - \log^{k-t-r} \left( \frac{n}{m} \right) \right]+O \left( \frac{\log^{k}(N)}{N} \right)\notag\\
&=\sum_{r=1}^{k+1} \frac{2^{k-r}}{r!} \left[ \log^{r}\left( \frac{n}{m} \right) - \log^{r}\left( \frac{m}{n} \right) \right] +  \sum_{\ell=0}^{k} \frac{\gamma_\ell}{\ell!} \sum_{r=0}^{k-\ell} \frac{2^{\ell+r}  \left( \log^{k-\ell-r} \left( \tfrac{m}{n} \right) - \log^{k-\ell-r} \left( \tfrac{n}{m} \right) \right)}{(k-\ell-r)!}+O \left( \frac{\log^{k}(N)}{N} \right)\notag\\
		&= (\gamma-1) \sum_{r=0}^{k} \frac{2^{k-r}}{r!} \left[ \log^{r}\left( \frac{m}{n} \right) - \log^{r}\left( \frac{n}{m} \right) \right] - \frac{\left( \log^{k+1}\left( \frac{m}{n} \right) - \log^{k+1}\left( \frac{n}{m} \right) \right)}{2(k+1)!}  \notag \\ 
		&\quad + \sum_{\ell=1}^{k} \frac{\gamma_\ell}{\ell!} \sum_{r=0}^{k-\ell} \frac{2^{k-r} }{r!} \left[ \log^{r} \left( \frac{m}{n} \right) - \log^{r} \left( \frac{n}{m} \right) \right]+O \left( \frac{\log^{k}(N)}{N} \right) \notag \\
		&= \sum_{\ell=0}^{k} \frac{a_\ell}{\ell!} \sum_{r=0}^{k-\ell} \frac{2^{k-r} }{r!} \left[ \log^{r} \left( \frac{m}{n} \right) - \log^{r} \left( \frac{n}{m} \right) \right] - \frac{\left( \log^{k+1}\left( \frac{m}{n} \right) - \log^{k+1}\left( \frac{n}{m} \right) \right)}{2(k+1)!} +O \left( \frac{\log^{k}(N)}{N} \right), \label{RHScalc}
	\end{align}
	where 
				\begin{align*}
					a_\ell:= 
					\begin{cases}
							\gamma -1  &  \text{if  } \ell=0, \\
							\gamma_{\ell}  & \text{if  } \ell>0 .
						\end{cases}
				\end{align*}
	From \eqref{maindiff1}, \eqref{RHScalc}, letting $N\to\infty$, rearranging and simplifying, we arrive at
	\small\begin{align*}
		&\sum_{\ell=0}^{k}\frac{1}{\ell!} \sum_{r=0}^{k-\ell}  \frac{1}{r!} \left\{ \frac{n(-1)^\ell 2^{k-\ell}}{m} \log^{\ell} \left( \frac{n}{m} \right)  \sum_{j=1}^{\infty} \left(  \psi_{r}' \left(1 + \frac{nj}{m}\right) - \frac{\log^{r}(\frac{nj}{m})}{\frac{nj}{m}} \right)  + a_\ell 2^{k-r} \log^{r} \left(\frac{n}{m} \right)  \right\} - \frac{\log^{k+1}\left(\frac{n}{m} \right)}{2(k+1)!}   \notag\\
		&=\sum_{\ell=0}^{k}\frac{1}{\ell!} \sum_{r=0}^{k-\ell}  \frac{1}{r!} \left\{ \frac{m(-1)^\ell 2^{k-\ell}}{n}  \log^{\ell} \left( \frac{m}{n} \right)  \sum_{j=1}^{\infty} \left(  \psi_{r}' \left(1 + \frac{mj}{n}\right) - \frac{\log^{r}(\frac{mj}{n})}{\frac{mj}{n}} \right)  + a_\ell 2^{k-r} \log^{r} \left(\frac{m}{n} \right)  \right\} - \frac{\log^{k+1}\left(\frac{m}{n}\right)}{2(k+1)!} . 
	\end{align*}
\normalsize
	Let $\frac{n}{m}=\alpha$ and $\frac{m}{n}=\beta$ to finally obtain \eqref{curious eqn} for $\alpha,\beta \in \mathbb{Q^+}$. Arguing as in the proof of Theorem \ref{guinand gen psi_k z>2}, we see that the result holds for $\alpha, \beta>0$.
\end{proof}

	\section{Concluding remarks}
	
	The results of this paper together fulfill the objective of deriving modular relations of two types for the generalized digamma functions $\psi_k(x)$ and their higher derivatives - one involving infinite series of $\psi_j^{(m)}(x), 0\leq j\leq k, k\in\mathbb{N}\cup\{0\}, m\in\mathbb{N}\cup\{0\}$, and the other involving finite sums of these functions. It is quite pleasing to know that none of these identities have lost their symmetrical nature. 
	
	The ones of the first type, that is, those involving infinite series, can be obtained from the second in the following way. We demonstrate this by deriving Theorem \ref{ramanujan generalization psi_k infinite} from Theorem \ref{finitetrans2}. Replace $m$ and $n$ by $mN$ and $nN$ respectively in Theorem \ref{finitetrans2} and let $x=\frac{1}{N}\left(\frac{1}{m}+\frac{1}{n}\right)$. Subtract the requisite number of terms from the asymptotic expansion of  $\psi_{\ell}$ to make the series involving $\psi_{\ell}$ convergent upon letting $N\to\infty$.  Finally, let $\alpha=m/n$ and $\beta=n/m$ and use analytic continuation to derive Theorem \ref{ramanujan generalization psi_k infinite} for Re$(\alpha)>0$ and Re$(\beta)>0$. However, this method can give us only the modular relation, not the integral involving the Riemann $\Xi$-function linked to it. See, for example, Theorem \ref{ramanujan analogue psi_1 infinite}. This linkage is extremely useful, for example, Hardy's proof \cite{Har} of the infinitude of the non-trivial zeros of $\zeta(s)$ on the critical line  - a first ever result of its kind - crucially employs the following result
	\begin{align*}
		\sqrt{\alpha}\bigg(\frac{1}{2\alpha}-\sum_{n=1}^{\infty}e^{-\pi\alpha^2n^2}\bigg)=\sqrt{\beta}\bigg(\frac{1}{2\beta}-\sum_{n=1}^{\infty}e^{-\pi\beta^2n^2}\bigg)=\frac{2}{\pi}\int_{0}^{\infty}\frac{\Xi(t/2)}{1+t^2}\cos\bigg(\frac{1}{2}t\log \a\bigg)\, dt.
	\end{align*}  
	
	This is why we derive Theorem \ref{ramanujan generalization psi_k infinite} from equation \eqref{mainneq2} as well. The reason we give the $\Xi$-function integral corresponding to only the modular relation in Theorem \ref{ramanujan analogue psi_1 infinite} and not for the general one in Theorem \ref{ramanujan generalization psi_k infinite} is because, in general, it is complicated in appearance. In principle, however, one can explicitly write it down. 
	
	Another way to derive the first equality in Theorem \ref{ramanujan analogue psi_1 infinite} could be to use the Kloosterman representation for $\psi(x)+\frac{1}{2x}-\log(x)$ \cite[p.~201, formula 5.74]{ober} and its analogue
		\begin{align*}
		\psi_1(x)+\frac{\log(x)}{2x}-\frac{1}{2}\log^{2}(x)=\frac{1}{2\pi i}\int_{d-i\infty}^{d+i\infty}\frac{\pi\zeta(1-s)}{\sin(\pi s)}\left(\gamma-\log(x)+\psi(s)\right)x^{-s}\, ds,
	\end{align*}
	where $d=\textup{Re}(s)>1$, which can be derived from \cite[Theorem 3.3]{bdg_log}. However, this is not a Mellin transform, rather a sum of two, which makes the derivation of the first equality in Theorem \ref{ramanujan analogue psi_1 infinite} difficult this way.
	
	One might as well derive Theorem \ref{ramanujan analogue psi_1 infinite} by starting with $\mathscr{I}(\alpha)$ and showing it to be equal to $\mathscr{F}_1(\alpha)$ or $\mathscr{F}_1(\beta)$ as done, for example, with the $\Xi$-function integrals in \cite{transf} although this looks difficult and unnatural because of the form of the integral. Moreover, it is unclear why it leads us to an interesting representation in terms of infinite series of $\psi$ and $\psi_1$. It was only because of the availability of \eqref{mainneq2} that we were able to get to the integral in this form.
	
	We remark that Theorems \ref{ramanujan generalization psi_k infinite}, \ref{ramanujan analogue psi_1 infinite} are actually valid for $\alpha, \beta\in\mathfrak{D'}$, where $\mathfrak{D'}$ is defined in \eqref{ddash}. This is now justified. From \cite[Equations (2.8), (2.11)]{ishibashi}, we have
	\begin{align}
	\psi_j(x)+\frac{\log^j(x)}{2 x}- \frac{\log^{j+1}(x)}{j+1}=(-1)^{j+1}\int_{0}^{\infty}e^{-xt}\left(\frac{1}{e^{t}-1}-\frac{1}{t}+\frac{1}{2}\right)S_{j+1}(t)\, dt,
	\end{align}
	where 	$S_j(t):=\sum_{m=0}^{j-1}a_{j,m}\log^{m}(t)$, 
	where, with $a_{1,0}=1$ and $a_{j, j-1}=1$, $a_{j, m}$ are recursively defined by
	$a_{j, m}=-\sum_{r=0}^{j-2}\binom{j-1}{r}\G^{(j-r-1)}(1)a_{r+1, m}, 0\leq m\leq j-2$. Employing the extension of Watson's lemma given in \cite[p.~14, Theorem 2.2]{temme2015} with the choice of $\alpha$ and $\beta$ in this extension to be $\alpha=-\pi/2$ and $\beta=\pi/2$, and the fact that $\psi_j(x)$ is analytic in $x\in\mathfrak{D'}$, we see that \eqref{asymptotic coffey} holds for $x\in\mathfrak{D'}$. (For $j=1$, this has been done in \cite[Theorem 3.2]{bdg_log}.) Hence the series in Theorems \ref{ramanujan generalization psi_k infinite} and \eqref{ramanujan analogue psi_1 infinite} are analytic functions of $\alpha$ and $\beta$ in this region. Thus, the results hold for $\alpha, \beta\in\mathfrak{D'}$. Also, the $\Xi$-function integral in Theorem \ref{ramanujan analogue psi_1 infinite} is analytic, as a function of $\alpha$, in  $\mathfrak{D'}$, as can be seen by invoking \eqref{strivert}, elementary bounds on zeta and hyperbolic functions. 
	It should also be possible to extend Theorems
	\ref{guinand gen psi_k z>2} and \ref{curious}  for $\alpha, \beta\in\mathfrak{D'}$, however, that would require first extending Lemma \ref{8.2a} for $x\in\mathfrak{D'}$. 
	
	It would be fascinating to see if there is a $\Xi$-function integral linked to the modular relations in Theorems \ref{guinand gen psi_k z>2} and \ref{curious}. We have no idea how the form of such integral would be, if at all it exists. 

It can be seen that the first equality of \eqref{mainneq2} is valid in the more general region Re$(z)>0, z\neq 2$. The restriction $0<\textup{Re}(z)<2$ has to be put only when we consider modular relation along with the $\Xi(t)$-function integral. It would be interesting to use \eqref{mainneq2} with $z=2+it, t\neq0$ and $\alpha\neq\beta$ to prove $\zeta(1+it)\neq0$. It is, of course, well-known that the latter is equivalent to the prime number theorem.

Recently, Darses and Hillion \cite{dh} applied an identity of Ramanujan \cite[Equation (22)]{riemann} involving the integral in \eqref{w1.26} towards obtaining closed-form identities for the polynomial moments with a weighted zeta square measure on the critical line. This certainly merits a similar study of the integral in \eqref{mathscr i alpha}.

Finally, it would be definitely worth seeing number-theoretic/combinatorial applications of Theorem \ref{Inv} different from what we have shown here.
	\begin{center}
		\textbf{Acknowledgements}
	\end{center}
	The authors thank Projesh Nath Choudhury of IIT Gandhinagar for informing them about the reference \cite{merca}. The first author is supported by the Swarnajayanti Fellowship grant SB/SJF/2021-22/08 of SERB (Govt. of India). The first and second authors are supported by the SERB CRG Grant CRG/2020/002367. They sincerely thank these agencies for the support. All of the authors thank IIT Gandhinagar for supporting their research.

\end{document}